\documentclass[a4paper,11pt]{amsart}
\usepackage{hyperref,latexsym}
\usepackage{enumerate}
\usepackage{graphicx}
\usepackage{color}
\usepackage{comment}
\usepackage[sort,nocompress]{cite}
\usepackage{hyperref}
\usepackage{tikz}
\usepackage{tikz-cd}
\usepackage{caption}
\usepackage{tkz-euclide}
\usepackage[toc,page]{appendix}

\theoremstyle{plain}
\newtheorem{theorem}{Theorem}[section]
\newtheorem{lemma}[theorem]{Lemma}

\newtheorem{proposition}[theorem]{Proposition}
\theoremstyle{definition}
\newtheorem{definition}[theorem]{Definition}

\theoremstyle{remark}
\newtheorem*{remark*}{Remark}

\numberwithin{equation}{section}

\newcommand{\powbra}[1]{^{(#1)}}

\newcommand{\R}{\mathbb{R}}
\newcommand{\Z}{\mathbb{Z}}
\newcommand{\T}{\mathbb{T}}
\newcommand{\Sph}{\mathbb{S}}
\newcommand{\set}[1]{\left\lbrace#1\right\rbrace}
\newcommand{\inprod}[2]{\left\langle#1,#2\right\rangle}
\newcommand{\inv}{^{-1}}
\newcommand{\subspaceleq}[1]{\underset{#1}{<}}

\makeatletter
\newcommand{\myitem}[1]{%
\item[#1]\protected@edef\@currentlabel{#1}%
}
\makeatother

\begin{document}

\title[]
      {Locally maximizing orbits for multi-dimensional twist maps and Birkhoff billiards}

\date{}
\author{Misha Bialy}
\address{School of Mathematical Sciences, Raymond and Beverly Sackler Faculty of Exact Sciences, Tel Aviv University,
Israel} 
\email{bialy@tauex.tau.ac.il}
\thanks{MB is partially supported by ISF grant 580/20,  DT is supported by ISF grants 580/20, 667/18 and DFG grant MA-2565/7-1  within the Middle East Collaboration Program.}

\author{Daniel Tsodikovich}
\address{School of Mathematical Sciences, Raymond and Beverly Sackler Faculty of Exact Sciences, Tel Aviv University,
	Israel}
\email{tsodikovich@tauex.tau.ac.il}


\begin{abstract}
In this work, we consider the variational properties of exact symplectic twist maps $T$ that act on the cotangent bundle of a torus or on a ball bundle over a sphere.
An example of such a map is the well-known Birkhoff billiard map corresponding to smooth convex hypersurfaces.
 In this work, we will focus on the important class $\mathcal{M}$ of orbits of $T$ which are locally maximizing with respect to the variational principle associated with a generating function of the  symplectic twist map.
 Our first goal is to give a geometric and variational characterization for the orbits in the class $\mathcal M$.

The billiard map is known to have two different generating functions, and this invokes the following question: how to compare the properties of these two generating functions? 
While our motivation comes from billiards, we will work in general, and assume that a general twist map $T$ has two generating functions.
Thus, we consider the orbits of $T$ which are locally maximizing with respect to either of the generating functions.
We formulate a geometric criterion guaranteeing that two generating functions of the same twist map have the same class of locally maximizing orbits, and we will show that the two generating functions for the Birkhoff billiard map do, in fact, satisfy this criterion.
The proof of this last property will rely on the Sinai-Chernov formula from geometric optics and billiard dynamics.
\end{abstract}

\maketitle

\section{Introduction and the results}\label{section:intro}

Symplectic twist maps of cotangent bundles often arise in mathematics and physics as maps describing conservative dynamical systems, and hence they are studied extensively. 
An exact symplectic twist map $T$ is one that has a generating function, $H$ (for the exact definitions, see Subsection \ref{subsection:twistmapintro}).
This allows studying the twist map using a variational approach: 
one can study the extremals of the formal functional \cite {AubryS1983TdFm, MatherJohnN1991Vcoo, Bangert1988MatherSF, MacKayR.S1985CKTa, MacKayRS1989CKtf}
\[\sum_{n=-\infty}^\infty H(q_n,q_{n+1}).\]

While this infinite sum may not be well-defined, its derivatives with respect to every $q_n$ are well defined.
Our main motivation is the Birkhoff billiard map in smooth convex hypersurfaces of the Euclidean space.
In this system, a particle moves in a straight line until it hits the boundary of the domain, at which point it switches direction according to the classical law of geometric optics: The angle of incidence is equal to the angle of reflection, and the incoming ray, outgoing ray, and the normal are all co-planar.
 See Figure \ref{fig: birkBilliard}.
   \begin{figure}
        \centering
       \begin{tikzpicture}[scale = 2]
   \draw[domain = -180:180, smooth, variable = \t, black, line width = 1.5pt] plot({cos(\t)},{sin(\t)});
   \tkzDefPoint(1,0){A};
   \tkzDefPoint(0,1){B};
   \tkzDrawPoints(A,B);
   \tkzDrawSegment[blue, line width = 1.5pt](A,B);
   \draw[dashed, line width = 1pt](1,-0.7)--(1,0.7);
   \draw[dashed, line width = 1pt](-0.7,1)--(0.7,1);
   \draw[blue](0,1)--(-0.3,0.7);
   \node[magenta] at (0.8, 0.4) {\footnotesize{$\alpha$}};
   \node[olive] at (0.4,0.8) {\footnotesize{$\alpha'$}};
   \node[olive] at (-0.4,0.8) {\footnotesize{$\alpha'$}};
       \end{tikzpicture}
       \caption{Birkhoff billiard dynamical system.\label{fig: birkBilliard}}
    \end{figure}
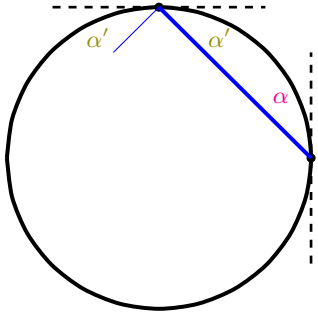

Variational study of the billiard dynamics in higher dimensions is an active field of research.
The behavior of the high-dimensional billiard system seems to be more flexible and varied compared to the two-dimensional system.
See \cite{1930-5311_2009_1_51} and \cite{TuraevClarke2020} for some results on the variational properties of billiards in higher dimensions, and see \cite{arnaud2010green} for a study of the variational properties of twist maps in higher dimensions.

The billiard map can be described using two different natural generating functions that correspond to two natural sets of symplectic coordinates on the space of oriented lines (see Subsection \ref{subsection:generatingFunctionsBilliards}):
 The standard generating function---the length of the chord, used by Birkhoff and others--- and the non-standard generating function (which is related to the support function), used recently in \cite{BialyMisha2020DRia, BialyMisha2017Abaa, BialyMisha2018Gbti}.
In light of this, a natural question arises: How to compare the variational properties of the two different variational principles?
While the billiard map is our main motivation, we will develop the theory for twist maps in general.
We will study the locally maximizing orbits of twist maps, which we call \textit{m-orbits}:
they correspond to the configurations $\set{q_n}$ for which every finite subsegment $\set{q_n}_{n=N}^M$ is a local maximum of the truncated action functional,
\[F_{M,N}(x_{M},...,x_{N})=H(q_{M-1},x_{M})+\sum_{i=M}^{N-1}H(x_{i},x_{i+1})+H(x_{N},q_{N+1}).\]
For more details on m-orbits, see Subsection \ref{subsubsection:morbits} below. The set $\mathcal M$, which is the set that consists of all
m-orbits, is a remarkable closed invariant set of a system.
 In particular, $\mathcal M$ contains all the invariant Lagrangian graphs, see \cite{MacKayRS1989CKtf,SEDP_1987-1988____A14_0} and Proposition \ref{prop:elementaryProperties} below.
This class is also useful when studying the rigidity of this system.
\begin{remark*}
	Traditionally, the standard generating function for billiards is the negative chord length.
	 However, we shall prefer, for convenience, the sign $+$ for the generating function and hence we use the twist condition with the opposite sign for the billiard map (that is, our twist is negative).
	  Consequently, we shall deal with maximizing (and not minimizing) orbits.
	   The sign of the twist condition in higher dimensions is not completely clear, see more in Subsection \ref{subsubsection:signtwist}.
\end{remark*} 

Our results are the following.
First, we summarize the ideas and methods of \cite{Bialy1993, Michael2012Hrfc, MacKayR.S1985CKTa} and we state the following criterion in terms of the Jacobi field.
One direction was shown in \cite{BialyML2004Stmw}, the converse is, to our knowledge, new (for the definition of matrix Jacobi field, see Subsection \ref{subsection:criteriontori}):
\begin{theorem}\label{thm:criterionMinMax}
Let $T$ be an exact symplectic twist map of the cotangent bundle of a torus.
The configuration $\set{q_n}$ is an m-configuration if and only if there exists a matrix Jacobi field $J_n$, for $n\in\Z$, of non-singular matrices, such that all the matrices 
\[X_n=-H_{12}(q_n,q_{n+1})J_{n+1}J_n\inv,\]
are symmetric negative definite.
\end{theorem}
Here and below, subscripts 1,2 denote partial derivatives with respect to the first or second variables, respectively.

We also generalize this theorem for twist maps of ball bundles over spheres, see Subsection \ref{subsubsection:twistballbundle} and Theorem \ref{prop:mOrbitCriterioBallBundle}.
As a result, we can continue to handle twist maps of cotangent bundles of torii and ball bundles over spheres in a unified way.

Next, we give a geometric criterion (which is inspired by the cone condition of \cite {MacKayR.S1985CKTa}) as a consequence  of Theorem \ref{prop:mOrbitCriterioBallBundle}, where we use the natural order on Lagrangian subspaces (for the definition of this order, see Subsection \ref{subsection:LagrangianSubspaces}):
\begin{theorem}\label{thm:morbitInequality}
An orbit $\set{(q_n,p_n)}$ is an m-orbit of a twist map $T$ if and only if there exists a field of Lagrangian subspaces $L_n$ along the orbit such that for all $n\in\Z$, the following three conditions hold:
\begin{enumerate}
\item\label{itm:transversality} $L_n$ is transversal to the vertical subspace at $(q_n,p_n)$, $V_n$.
\item\label{itm:invariance} $L_{n+1}=dT_{(q_n,p_n)}(L_n)$.
\item\label{itm:inequality}$\alpha_n \subspaceleq{V_n} L_n \subspaceleq{V_n}\beta_n$,
\end{enumerate}
where $\alpha_n$, $\beta_n$ are the image and preimage (respectively) of the vertical subspace: $\alpha_n = dT(V_{n-1})$, $\beta_n=dT\inv(V_{n+1})$, and $\subspaceleq{V_n}$ denotes the partial order on Lagrangian subspaces transversal to $V_n$ (see Subsection \ref{subsection:LagrangianSubspaces}).
\end{theorem}
Suppose that $T$ is an exact symplectic twist map with respect to two generating functions, denoted by $H^{(1)}$ and $H^{(2)}$.
Write $\mathcal{M}_{H^{(1)}}$ for the set that consists of all m-orbits of $H^{(1)}$, and $\mathcal{M}_{H^{(2)}}$ for the set that consists of all m-orbits of $H^{(2)}$.
Denote by $V^{H\powbra{i}}$ the vertical subspace of the symplectic coordinates for $H\powbra{i}$, and by $\alpha^{H\powbra{i}}$, $\beta^{H\powbra{i}}$ the image and preimage of $V^{H\powbra{i}}$ by $dT$.
We now formulate the Geometric Assumption \ref{GeometricAssumption} about the symplectic coordinates corresponding to the generating functions $H^{(1)}$ and $H^{(2)}$.
 This condition roughly says that the two sets of coordinates are not too far apart.
 \begin{itemize}
\myitem{(GA)}\label{FakeGeometricAssumption} Suppose that at every point of $\mathcal{M}_{H\powbra{1}}\cup\mathcal{M}_{H\powbra{2}}$ 
 there exists a homotopy of Lagrangian subspaces $\set{V_t}$ connecting $V^{H\powbra{1}}$ and $V^{H\powbra{2}}$, such that for all $t$, the subspace $V_t$ is transversal to all four subspaces $\alpha^{H\powbra{1}}$, $\beta^{H\powbra{1}}$, $\alpha^{H\powbra{2}}$, $\beta ^{H\powbra{2}}$.

\end{itemize} 
%

\begin{theorem}\label{thm:geometricAssumptionImpliesMOrbitEq}
Suppose that $H^{(1)}$ and $H^{(2)}$ are two generating functions of an exact symplectic twist map $T$.
If $H^{(1)}$ and $H^{(2)}$ satisfy the Geometric Assumption \ref{GeometricAssumption}, then the sets m-orbits coincide: 
\[\mathcal{M}_{H^{(1)}} = \mathcal{M}_{H^{(2)}}.\]
\end{theorem}
Finally, we show that the two generating functions of the billiard map satisfy this geometric assumption (for the definitions of those generating functions, see Subsection \ref{subsection:generatingFunctionsBilliards}).
To that end, we analyze the billiard reflection using the well-known Sinai-Chernov formula \cite{SinaiChernovFormula}.
\begin{theorem}\label{thm:equalityForBilliard}
Let $\Sigma\subseteq \R^d$ be a  $C^2$-smooth
convex hypersurface of positive curvature.
Let $L$ and $S$ denote the two generating functions for the Birkhoff billiard in $\Sigma$.
Then:
\[\mathcal{M}_L=\mathcal{M}_S.\]
\end{theorem}

\begin{remark*}
	Let us recall a result by Bernard \cite{10.1215/S0012-7094-07-13631-7} and by Mazzucchelli and Sorrentino \cite{MAZZUCCHELLI2016419}  on Tonelli Hamiltonians. 
It was shown that if a Tonelli Hamiltonian remains Tonelli after an exact symplectic change of variables, then the Aubry, Ma\~{n}e, and Mather sets are the same for both Hamiltonians.	
Every exact symplectic twist map of a cylinder can be seen as a time-one map for some Tonelli Hamiltonian \cite{moser_1986}. 
This suggests that in the two-dimensional case, the result on Tonelli Hamiltonians and our results may be connected.
Additionally, it is known that if a higher dimensional twist map satisfies a ``symmetric" twist condition, then it is also a time-one map of a Tonelli Hamiltonian.
However, this interpolation result is not known for Birkhoff billiards in high dimensions.
Our approach in this work is direct, meaning that we deal directly with the discrete system, and do not rely on interpolation by Tonelli Hamiltonians.
\end{remark*}

\subsection*{Structure of the paper} In Section \ref{section:bg}, we recall the basic definitions of twist maps, and generating functions, and we also recall what the generating functions $L,S$ for the Birkhoff billiard map are.
In Section \ref{section analysis}, we prove Theorem \ref{thm:criterionMinMax}.
In Section \ref{section:geometricMeaning}, we formulate the Geometric Assumption \ref{GeometricAssumption} and prove Theorems \ref{thm:morbitInequality}, \ref{thm:geometricAssumptionImpliesMOrbitEq}.
Lastly, in Section \ref{section:highDimBilliard}, we prove Theorem \ref{thm:equalityForBilliard} by showing that the two generating functions of the billiard map satisfy the assumptions of Theorem \ref{thm:geometricAssumptionImpliesMOrbitEq}.
Appendix \ref{app:secondOrderDerivatives} contains the derivatives for the generating functions $L,S$ of billiards.
 Appendix \ref{app:sinaiChernovProof} contains a proof of the Sinai-Chernov formula.
\\
\subsection*{Acknowledgements} It is a pleasure to thank Leonid Polterovich and Lev Buhovsky for encouraging discussions.

\section{Preliminaries}\label{section:bg}
In this section, we recall some of the definitions and constructions that will be used in this work.
In Subsection \ref{subsection:twistmapintro}, we recall the basic definitions of twist maps of cotangent bundles of tori and of ball bundles over spheres.
In Subsection \ref{subsection:generatingFunctionsBilliards}, we focus on the Birkhoff billiard map, and present two  generating functions for it.
In Subsection \ref{subsection:LagrangianSubspaces}, we recall some properties and constructions that are related to Lagrangian subspaces of symplectic vector spaces.
Finally, in Subsection \ref{subsection:wavefront}, we recall the correspondence between Lagrangian subspaces of the space of oriented lines and wave fronts.
\subsection{Twist maps of cotangent bundles of tori and ball bundles over spheres}\label{subsection:twistmapintro}
\subsubsection{Twist maps of cotangent bundles of tori}\label{subsubsection:twisttori}
Symplectic twist maps of cotangent bundles of tori arise in various mathematical and physical settings as maps that describe conservative systems.
Let us recall the definition.
\begin{definition}\label{def:twistMapDef}
Let $T$ be a diffeomorphism of the (co-)tangent bundle of the $d$ dimensional torus, $T^*\T^d\cong \T^d\times \R^d$.
The map $T$ is called an \textit{exact symplectic twist map} if it can be lifted to a diffeomorphism 
\[(q,p)\mapsto(Q(q,p),P(q,p))\]
 of $\R^{2d}$ which satisfies:
\begin{enumerate}
\item $Q(q+m,p)=Q(q,p)+m$ for all $m\in \Z^d$.
\item\label{itm:twistDiffeo}For every fixed $q$, the map $p\mapsto Q(q,p)$ is a diffeomorphism.
\item\label{itm:genFun} The 1-form $PdQ-pdq$ is exact, meaning that there exists a function $H:\R^{2d}\to\R$ (called a \textit{generating function}) which satsifies
\[PdQ-pdq = dH(q,Q),\]
and in addition $H(q+m,Q+m)=H(q,Q)$ for all $m\in\Z^d$.
\end{enumerate}
\end{definition}
Condition \ref{itm:genFun} implies that the map $T$ preserves the standard symplectic form $dp\wedge dq$.
The generating function allows to investigate the twist map using variational approaches, as it was done in \cite{AubryS1983TdFm,Bangert1988MatherSF,MacKayR.S1985CKTa, MacKayRS1989CKtf, MatherJohnN1991Vcoo}.
A function $H$ which satisfies the assumption $H(q+m,Q+m)=H(q,Q)$ for all $m\in\Z^d$, and also satisfies the assumption that the matrix of mixed partial derivatives $H_{12}$ is non-degenerate (the twist condition) can be used to define an exact symplectic twist map.
This is done by setting $F(q,p)=(Q,P)$ if $p=-H_1(q,Q)$ and $P=H_2(q,Q)$, and then letting $T$ be the projection of $F$ to $T^*\T^d$.
The condition that $H_{12}$ is non-degenerate guarantees that this definition does, in fact, define $(Q,P)$ as a function of $(q,p)$.
\subsubsection{The sign of a twist maps}
\label{subsubsection:signtwist}
In the case that $d=1$, twist maps can be assigned a sign in the following way: Since $H_{12}$ is non-degenerate, it is a non-zero scalar function.
Therefore, $H_{12}$ is either positive or negative, and the sign of the twist map can be defined according to it.
The map has a positive twist if $H_{12}<0$ (this is equivalent to $\frac{\partial Q}{\partial p} > 0$: hence, word ``positive" is used here).
The sign of the twist map plays a significant role. 
For example, all Aubry-Mather sets and rotational invariant curves consist of minimizing orbits for positive twist maps.
In the case that $d>1$, it seems that the notion of a signed twist map is more mysterious.
One can define the sign of the twist map by considering the sign of the symmetric matrix $H_{12}+H_{12}^T=H_{12}+H_{21}$.
In this case, it was shown in \cite{MacKayRS1989CKtf} again that Lagrangian graphs of positive twist maps are composed of minimizing orbits.
There is a different definition for the sign by Herman (see \cite{SEDP_1987-1988____A14_0}), which is not equivalent to the previous one.
 In fact, as proved in \cite{MacKayRS1989CKtf}, a more general condition of superlinearity of the generating function, which unifies both definitions, implies the minimality property of invariant Lagrangian graphs (see also \cite[Theorem 35.2]{gole2001symplectic}).

For twist maps of ball bundles over spheres (see Subsection \ref{subsubsection:twistballbundle}), the notion of a signed twist map is even more mysterious, and the definitions mentioned above do not adapt easily to that setting. 
Nonetheless, the notion of signed twist map does seem to exist, because for convex billiards we can prove the results for maximizing orbits, but not for minimizing orbits.
In any case, the sign of the twist map will not play a significant role in the general theory that we develop here. 
\subsubsection{Twist maps of ball bundle over sphere}\label{subsubsection:twistballbundle}
In this subsection, we briefly explain how the notion of a twist map of a torus can be adapted to ball bundles over a sphere (cf. \cite{gole2001symplectic}).
The billiard ball map in higher dimension is our main example of such a map, and the one we will discuss in detail later.
\begin{definition}\label{def:twistMapBallBundle}
Let $N\subseteq\mathbb{R}^d$ be a submanifold diffeomorphic to $\Sph^{d-1}$. 
Suppose that  $\pi:M\rightarrow N$ is a ball bundle over $N$, where $M\subseteq T^* N$ is a neighborhood of the zero section of $T^*N$, such that for every $x\in N$ the fiber  $N_x=\pi\inv(x)\subset T^*_xN$ is diffeomorphic to a $d$-dimensional ball. 
\begin{enumerate}
\item A diffeomorphism $T:M\to M$ is called a \textit{twist map} of $M$, if it is a symplectomorphism, and it satisfies the \textit{twist condition}: 
For all $x\in N$, the map $\pi\circ T\mid _{N_x}:N_x\to N$ is  a diffeomorphism from the interior of the ball $N_x$ to $N\setminus\set{x}$.
\item  
A \textit{generating function} for a twist map $T$ is a function $H:N\times N\setminus \Delta \to \R$ (where $\Delta$ denotes the diagonal) such that for all $x,y\in N$ and $u\in N_x$, $v\in N_y$ it holds that:
\[T(x,u)=(y,v)\Longleftrightarrow\begin{cases}u = -H_1(x,y) \\
	v = H_2(x,y), \\
\end{cases}\]
where subindices  $1,2$ stand for the differentials with respect to $x,y$.
\end{enumerate}

\end{definition}

We equip $N$ with a Riemannian metric induced from $\mathbb{R}^d$.
Then we can view the second order derivatives of $K$ as linear operators:
\begin{gather*}H_{12}(x,y):T_y N\to T_x N,\quad H_{21}(x,y):T_x N\to T_y N,\\
	H_{11}(x,y):T_x N\to T_x N, \quad H_{22}(x,y): T_y N\to T_y N,
\end{gather*}
where the operators $H_{12}$, $H_{21}$ are conjugate to each other, and $H_{11}$, $H_{22}$ are self-adjoint.
In this setting, the twist condition is the requirement that the operators $H_{12}$, $H_{21}$ are non-degenerate.
The Riemannian metric on $N$ allows us to identify the tangent and the cotangent spaces to $N$, and we will do so freely.
\subsubsection{Class of m-orbits for twist maps}\label{subsubsection:morbits}
The following discussion applies to both, twist maps of tori and twist maps of ball bundles over spheres.
Suppose that $T$ is an exact symplectic twist map with a generating function $H$.
The function $H$ gives rise to the following formal variational functional:
\[\sum_{n=-\infty}^\infty H(q_n,q_{n+1}).\]
It is straightforward to see that a sequence $\set{q_n}$ is a configuration (that is, it is the projection of an orbit of $T$, $\set{(q_n,p_n)}$) if and only if it is a critical point of the variational functional, i.e., any finite subsegment $\set{q_n}_{n=M}^N, M\leq N$ is a critical point of the truncated functional 
\[F_{MN}(x_{M},...,x_{N})=H(q_{M-1},x_{M})+\sum_{i=M}^{N-1}H(x_{i},x_{i+1})+H(x_{N},q_{N+1}).\]
In particular, sometimes it is interesting to look for sequences that are not only critical points but are actual local minimizers or maximizers.
At such points, the matrix of the second variation needs to be positive or negative definite, respectively.
It is straightforward to check that the second variation of $F_{MN}$ has the following block form:
\begin{equation}\label{eq:secondVariation}
W_{MN} =\delta^2 F_{MN} =  \begin{pmatrix}
a_{M} & b_{M} & 0\cdots&0 &0 \\
b^*_{M}& a_{M+1} &b_{M+1}\cdots&0&0  \\
\vdots  & \ddots  & \ddots & \ddots & \vdots \\
0 & \cdots&b^*_{N-2}&a_{N-1}&b_{N-1}  \\
0 & \cdots& 0&b^*_{N-1} & a_{N} 
\end{pmatrix},
\end{equation}
where the operators  $a_n,b_n$ are given by
\[a_n = H_{11}(q_n,q_{n+1})+H_{22}(q_{n-1},q_n),\quad b_n = H_{12}(q_n,q_{n+1}), \quad b_n^* = H_{21}(q_n,q_{n+1}),\]
where $b_k^*$ denotes the dual linear map to $b_k$.

In this work, we will consider locally maximizing orbits, and we call such orbits \emph{m-orbits} or \emph{m-configurations} (depends on whether we consider the point $(q_n,p_n)$, or its projection, $q_n$, respectively).
As mentioned in Proposition \ref{prop:elementaryProperties}, for negative twist maps, all orbits lying on an invariant Lagrangian graph are m-orbits.
\subsection{Generating functions for Birkhoff billiards}\label{subsection:generatingFunctionsBilliards}
Let $\Sigma$ be a smooth convex hypersurface in $\R^d$. 
The dynamical billiard system in $\Sigma$ is defined as the motion of a particle moving in the direction of a unit vector $v$.
When the particle hits the boundary, it reflects according to the classical law of geometric optics. 
It will move in the direction of a unit vector $u$ which satisfies the following two conditions: first, $u$, $v$ and the normal $n$ to $\Sigma$ at the collision point are co-planar; and second, $u$ and $v$ make equal angles with the tangent space to $\Sigma$ at the collision point, $\set{n}^\perp$.

The phase space of the billiard map is the space of oriented lines in $\R^d$ that intersect the hypersurface $\Sigma$.
Denote this space by $\mathcal{L}$. 
The space $\mathcal{L}$ can be seen as a ball bundle over a hypersurface $N$ diffeomorphic to a sphere in two different ways.
Thus, we have two different sets of symplectic coordinates on $\mathcal{L}$ inherited from $T^*N$ 
 (see Figure \ref{fig:twoGeneratingFunctions}): L-coordinates and S-coordinates.

\textbf{L-coordinates:} Pick a point $x\in \Sigma$ and let $w\in T_x \Sigma$ be a vector with $|w|<1$.
The pair $(x,w)$ determines the oriented line which passes through $x$ and has the direction $w-\sqrt{1-|w|^2}n_x$ where $n_x$ is the outer normal to $\Sigma$ at $x$.
In these coordinates, $wdx$ is a primitive of the canonical symplectic form of $T^* \Sigma$.
Thus, in this description, $\mathcal{L}$ is the unit ball bundle of $N=\Sigma$.
The generating function in those coordinates is 
\[L:(\Sigma\times \Sigma \setminus \Delta)\to\R,\quad L(x,y)=|x-y|.\]
This generated function was studied and used by Birkhoff.

\textbf{S-coordinates:} Assume that the origin is inside $\Sigma$. 
Given an oriented line in $\R^d$, we can associate to it a point in $T^*\Sph^{d-1}$ in the following way: Take the unit vector direction of the line, $u\in\Sph^{d-1}$, and take $v\in\set{u}^\perp \cong T^*_u\Sph^{d-1}$ to be the intersection point of the line with $\set{u}^\perp$ (where we identify vectors with co-vectors using the inner product), and set $p=-v$.
Then the form $pdu$ is also a primitive of the canonical symplectic form on $T^* \Sph^{d-1}$. 
This symplectic form coincides with the one described 
for the L-coordinates.
 Thus, in this case, we have $N=\Sph^{d-1}$, and the ball $N_u$, for $u\in \Sph^{d-1}$, is (the negative of) the projection of $\Sigma$ on the hyperplane $\set{u}^\perp$.

The generating function for the billiard map in those coordinates is $S:(\Sph^{d-1}\times \Sph^{d-1}\setminus \Delta)\to\R$ defined by 
\[S(u_1,u_2)=\inprod{G\inv\Big(\frac{u_1-u_2}{|u_1-u_2|}\Big)}{u_1-u_2},\]
 where $G:\Sigma\to \Sph^{d-1}$ denotes the Gauss map of $\Sigma$.
This generating function was used in ellipsoids in \cite{suris2016billiards}, and  in \cite{BialyMisha2020DRia, BialyMisha2017Abaa, BialyMisha2018Gbti, bialy2022locally} in the general case.

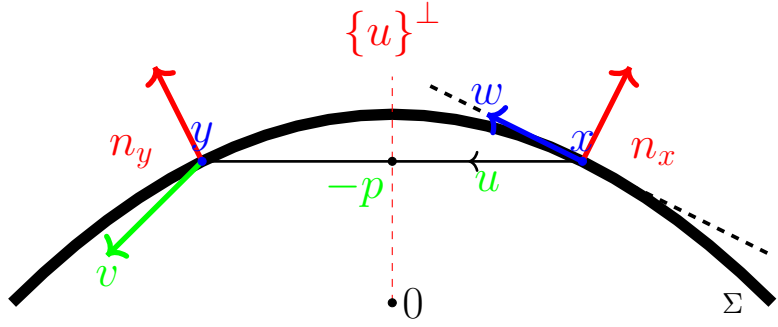
\begin{figure}
	\centering
	\begin{tikzpicture}[scale = 2.5]
	\begin{scope}[decoration={
markings,
mark = at position 0.3 with {\arrow{>}}}
]
	\draw[line width = 1.5mm, domain = -2:2] plot(\x,1-0.25*\x*\x);
\draw[black, line width = 1pt, postaction = {decorate}](1,0.75)--(-1,0.75);
\end{scope}
\draw[line width = 0.5mm, black, dashed](0.2,1.15)--(2,0.25);
\draw[red, dashed] (0,0)--(0,1.2);
	\begin{scope}[decoration={
markings,
mark = at position 1 with {\arrow{>}}}
]
\draw[blue, postaction = {decorate}, line width = 0.7mm](1,0.75)--(0.5,1);
\draw[red, postaction = {decorate}, line width = 0.7mm](1,0.75)--(1.25,1.25);
\draw[green,postaction = {decorate}, line width = 0.7mm](-1,0.75)--(-1.5,0.25);
\draw[red, postaction = {decorate}, line width = 0.7mm](-1,0.75)--(-1.25,1.25);
\end{scope}

\tkzDefPoint(1,0.75){A};
\tkzDrawPoint[blue](A);
\tkzDefPoint(-1,0.75){B}
\tkzDrawPoint[blue](B);
\tkzDefPoint(0,0){O};
\tkzDrawPoint[black](O);
\tkzDefPoint(0,0.75){C};
\tkzDrawPoint[black](C);

\node[above,blue] at (1,0.75) {\LARGE{$x$}};
\node[above,blue] at (-1,0.75) {\LARGE{$y$}};
\node[below, green] at (-0.2,0.75) {\LARGE{$-p$}};
\node[above, red] at (0,1.25) {\LARGE{$\set{u}^\perp$}};
\node[right, red] at (1.2,0.8) {\LARGE{$n_x$}};
\node[left, red] at (-1.2,0.8) {\LARGE{$n_y$}};
\node[green,below] at (0.5,0.75) {\LARGE{$u$}};
\node[blue, above] at (0.5,1) {\LARGE{$w$}};
\node[green, below] at (-1.5,0.25) {\LARGE{$v$}};
\node[right] at(0,0) {\LARGE{$0$}};
\node[left] at (1.9,0) {$\LARGE{\Sigma}$};
	\end{tikzpicture}
	\caption{Two sets of symplectic coordinates for the billiard map inside a hypersurface $\Sigma\subseteq\R^d$. \label{fig:twoGeneratingFunctions}}
\end{figure}
The following two lemmas state that these two functions are indeed generating functions for billiards, and that the momentum coordinate is in fact correct.
We also compute the second derivatives of these functions, which will be used in Section \ref{section:highDimBilliard}.
These formulas were already derived in \cite{BialyMisha2018Gbti, 1930-5311_2009_1_51}, but we include the computation for completeness.
We also present the formulas here in a slightly different form that will be more convenient for us in Section \ref{section:highDimBilliard}.

We use the following notation: For a unit vector $a$, $p_a$ denotes the orthogonal projection to $\set{a}^\perp$. 
For a pair of unit vectors $a$,$b$, $p_{a,b}$ denotes the restriction of $p_b$ to $\set{a}^\perp$.
Also, for two distinct unit vectors $a$, $b$, write $n(a,b)=\frac{a-b}{|a-b|}$.
\begin{lemma}\label{lem:partialDerivativesS}
The partial derivatives of $S$ are: (here $u_1$, $u_2$, $u_3$ denote distinct unit vectors, $n_3 = n(u_2,u_3)$, $n_2 = n(u_1,u_2)$, and $x_i$ are the points of the hypersurface where the unit normals are $n_i$)
\[S_1(u_2,u_3) = p_{u_2}G\inv(n(u_2,u_3)), \quad S_2(u_1,u_2) = -p_{u_2}G\inv(n(u_1,u_2)).\]
And the second-order derivatives are:
\begin{gather*}
S_{11}(u_2,u_3)=\frac{1}{|u_2-u_3|}p_{n_3,u_2}dG_{n_3}\inv p_{u_2,n_3} - \inprod{x_3}{u_2}I, \\
S_{22}(u_1,u_2)=\frac{1}{|u_1-u_2|}p_{n_2,u_2}dG_{n_2}\inv p_{u_2,n_2} + \inprod{x_2}{u_2}I.
\end{gather*}
And their sum is:
\begin{gather*}
S_{11}(u_2,u_3)+S_{22}(u_1,u_2)=\frac{1}{|u_1-u_2|}p_{n_2,u_2}dG_{n_2}\inv p_{u_2,n_2} + \\
+\frac{1}{|u_2-u_3|}p_{n_3,u_2}dG_{n_3}\inv p_{u_2,n_3} - |x_2-x_3|I.
\end{gather*}
In addition, the mixed partial derivative $S_{12}(u_1,u_2): T_{u_2}\Sph^{d-1}\to T_{u_1} \Sph^{d-1}$ is:
\[S_{12}(u_1,u_2)  = -\frac{1}{|u_1-u_2|} p_{n_2,u_1} dG\inv_{n_2} p_{u_2,n_2},\]
which is non-degenerate.
\end{lemma}
Notice that $p_{u_2}G\inv(n(u_2,u_3))$ and $p_{u_2}G\inv(n(u_1,u_2))$ are exactly the negatives of the momentum coordinates for $S$, so if we write $v_1 = -p_{u_2}G\inv(n(u_2,u_3))$ and $v_2 = -p_{u_2}G\inv(n(u_1,u_2))$, then we indeed have 
$v_1 = -S_1$, $v_2 = S_2$.

\begin{lemma}\label{lem:partialDerivativesL}
The partial derivatives of $L$ are: (here $x_1$, $x_2$, $x_3$ are points on the hypersurface, $n_i=G(x_i)$ is the outer normal at those points, and $u_1$ is the unit vector from $x_1$ to $x_2$ and $u_2$ is the unit vector from $x_2$ to $x_3$)
\[L_1(x_2,x_3) = -p_{n_2}u_2,\quad L_2(x_1,x_2)=p_{n_2}u_1.\]
The second-order partial derivatives are:
\begin{gather*}
L_{11}(x_2,x_3)=\frac{1}{|x_3-x_2|}I-\frac{\inprod{u_2}{\cdot}}{|x_3-x_2|}p_{n_2}u_2+\inprod{u_2}{n_2}dG_{x_2}, \\
L_{22}(x_1,x_2)=\frac{1}{|x_2-x_1|}I-\frac{\inprod{u_1}{\cdot}}{|x_2-x_1|}p_{n_2}u_1 - \inprod{u_1}{n_2}dG_{x_2}.
\end{gather*}
And their sum is:
\[L_{11}(x_2,x_3)+L_{22}(x_1,x_2)=\Bigg(\frac{1}{|x_1-x_2|}+\frac{1}{|x_2-x_3|}\Bigg)p_{u_2,n_2}p_{n_2,u_2} + 2\inprod{u_2}{n_2}dG_{x_2}.\]
In addition, the mixed partial derivative $L_{12}(x_2,x_3):T_{x_3}\Sigma\to T_{x_2}\Sigma$ is:
\[L_{12}(x_2,x_3) = -\frac{1}{|x_3-x_2|}p_{u_2,n_2}p_{n_3,u_2},\]
which is non-degenerate.
\end{lemma}
The proofs of these Lemmas are in Appendix \ref{app:secondOrderDerivatives}.

\subsection{Lagrangian subspaces}\label{subsection:LagrangianSubspaces}
In this subsection, we recall the construction of a partial order on the set of Lagrangian subspaces of a symplectic vector space that are transversal to a  given Lagrangian subspace.
This will be used in Section \ref{section:geometricMeaning}.
Consider a symplectic vector space $(U,\omega)$, and a Lagrangian subspace $\gamma$.
Choose symplectic coordinates $(q,p)$ on $U$ such that $\gamma = \set{q=0}$.
Then the partial order is defined as follows: If $L$ is a Lagrangian subspace of $U$ which is transversal to $\gamma$, then there exists a unique symmetric matrix $A$ such that $L=\set{p=Aq}$.
If $L_1$, $L_2$ are two Lagrangian subspaces that are transversal to $\gamma$, and $A_1$, $A_2$ are the respective symmetric matrices, then we set $L_1\subspaceleq{\gamma} L_2$ if and only if $A_1 < A_2$, meaning that $A_2-A_1$ is a positive definite matrix.

The obtained relation $\subspaceleq{\gamma}$ does not depend on the choice of symplectic coordinates as long as $\gamma$ has the description $\set{q=0}$ in those coordinates.
Indeed, if $(q',p')$ is another set of symplectic coordinates such that $\gamma=\set{q'=0}$, then the transition matrix between those coordinates has the form $$\begin{pmatrix}
A & 0 \\ B & C
\end{pmatrix},\quad A^TC=I,\quad A^TB=B^TA,$$
since it must be a symplectic matrix. 
From here, if a subspace transversal to $\gamma$ has the description $\set{p=Kq}$ in one set of coordinates and $\set{p'=K'q'}$ in the other set of coordinates, then it is simple to compute that $K'=C K C^T + B C^T$, and hence any inequality between subspaces will be independent of the coordinates.

This partial order is related to the index form (see \cite{ArnoldV.I.1985TSta}) defined as follows. 
Let $\alpha$,$\beta$ be two Lagrangian subspaces of $U$, which are transversal to each other (but might not be transversal to $\gamma$). 
Then any $u\in U$ can be uniquely written as $u=u_1+u_2$ for $u_1\in\alpha$ and $u_2\in\beta$, and we define a quadratic form $Q[\alpha,\beta]$ by:
\[Q[\alpha,\beta](u)=\omega(u_1,u_2).\]
Then the following holds:
\begin{proposition}\label{prop:orderQuadraticForm}
\begin{enumerate}
\item\label{itm:singleInequalityPositive} Suppose that $\alpha$ and $\beta$ are Lagrangian subspaces transversal to $\gamma$.
Then
\[\alpha\subspaceleq{\gamma} \beta\Longleftrightarrow Q[\alpha,\beta]\mid_\gamma >0.\]
\item\label{itm:doubleInequalityNegative} Suppose that $\alpha$, $\beta$, and $L$ are Lagrangian subspaces transversal to $\gamma$, and $\alpha\subspaceleq{\gamma} \beta$. 
Then:
\[\alpha\subspaceleq{\gamma} L \subspaceleq{\gamma} \beta \Longleftrightarrow Q[\alpha,\beta]\mid_L < 0.\]
\end{enumerate}
\end{proposition}
\begin{proof}[Proof of \ref{itm:singleInequalityPositive}]
Let $A$ and $B$ be the symmetric matrices such that $\alpha$ is the graph of $A$ and $\beta$ is the graph of $B$.
Take a vector $(0,z)\in\gamma$.
We write it as a sum of a vector in $\alpha$ and a vector in $\beta$:
\[(0,z)=(x,Ax)+(y,By),\]
from which we get that $x=-y$ and $z=(A-B)x$.
Now,
\[Q[\alpha,\beta](0,z)=\omega((x,Ax),(-x,-Bx))=\]
\[\inprod{Ax}{-x}-\inprod{x}{-Bx}=\inprod{x}{(B-A)x},\]
from which it follows that $Q[\alpha,\beta]\mid_\gamma>0$ if and only if $B-A$ is positive definite, and this holds if and only if $\alpha\subspaceleq{\gamma} \beta$.
\end{proof}

\begin{proof}[Proof of \ref{itm:doubleInequalityNegative}]
Applying a suitable linear symplectic transformation, we can assume:
\[
\gamma=\set{q=0},\quad\alpha=\set{p=0}.
\]
We write $\beta, L$ with the help of symmetric matrices
\[
\beta=\{p=Bq\}, \quad L=\{p=Wq\},\]
 where $B$ is positive definite, since $\alpha \subspaceleq{\gamma} \beta$.
Moreover, we have: 
\[
\alpha\subspaceleq{\gamma} L \subspaceleq{\gamma} \beta \iff 0<W<B.
\]
Next we compute $Q[\alpha,\beta]\mid_L.$
We split any vector 
\[(q,Wq)=(x,0)+(y,By)\,, \textrm{where}\   x+y=q,\ By=Wq.\]
We find
\[
y=B^{-1}Wq,\ x=q-B^{-1}Wq.
\]
Hence,
\begin{gather*}
\omega\left((x,0),(y,By)\right)=-\inprod{x}{By}=-\inprod{q}{Wq}+\inprod{B^{-1}Wq}{Wq}=\\
=-\inprod{W^{-1}z}{z}+\inprod{B^{-1}z}{z},\  \textrm{where}\  z:=Wq.
\end{gather*}
Thus, we see that the last expression is a negative definite quadratic form
$\iff {W^{-1}}>B^{-1}>0\iff 0<W<B.$
\end{proof}
In Section \ref{section:geometricMeaning}, we will consider homotopies of Lagrangian subspaces.
The following lemma states that the sign of the restriction of $Q[\alpha,\beta]$ to a Lagrangian subspace $\gamma$ is stable under homotopies of $\gamma$.
\begin{lemma}\label{lem:homotopyTransversalityPersistence}
Suppose that $U$ is a symplectic vector space, and that $\alpha$ and $\beta$ are transversal Lagrangian subspaces.
Suppose that $\set{\gamma_t}_{t\in[0,1]}$ is a homotopy of Lagrangian subspaces, such that for all $t\in[0,1]$, $\gamma_t$ is transversal to $\alpha$ and to $\beta$. 
If $Q[\alpha,\beta]$ is positive definite on $\gamma_0$, then it is positive definite on $\gamma_t$ for all $t\in[0,1]$.
\end{lemma}
\begin{proof} We apply a symplectic linear transformation to get $\alpha=\set{p=0}$, $\beta=\set{q=0}$.
 In this case, 
	$Q[\alpha,\beta]=-\inprod{p}{q}$.
	 Moreover, since $\gamma_t$ is transversal to $\beta=\set{q=0}$, we can write
	 \[\gamma_t=\{p=A_tq\}\]
	 for a symmetric matrix $A_t$.
	  Therefore, the quadratic form $Q[\alpha,\beta]|_{\gamma_t}$ is $-\inprod{A_tq}{q}$. 
	  For $t=0$, and hence also for small $t$, the matrix $A_t$ is negative definite.
	   Then this property persists, since if it is violated for the first time at some $t_*\in(0,1]$, then $A_{t_*}$ has a non-trivial kernel, meaning that $\gamma_{t_*}$ is not transversal to $\set{p=0}$.
	\end{proof}

\subsection{Wave Fronts}\label{subsection:wavefront}
In what follows, we will use the machinery of \textit{wave fronts}, which provides a convenient way to think about Lagrangian subspaces of the space of oriented lines in $\R^d$, $\mathcal{L}$. 
For completeness, we include a brief exposition of this construction.
For more details, see \cite{arnold2001singularities, SinaiChernovFormula,Chernov06chaoticbilliards, AST_2003__286__119_0,WojtkowskiMaciejP.2004HB}.

Start with $T\R^d \cong T^*\R^d \cong \R^{2d}$, equipped with coordinates $(q,p)$, and consider the Hamiltonian $H(q,p)=\frac{1}{2}|p|^2$.
Suppose that $K\subseteq T\R^d$ is a Lagrangian submanifold which is invariant under the Hamiltonian flow of $H$, lying in the energy level $\set{H=\frac{1}{2}}=\set{|p|=1}$, and suppose that $z_0=(q_0,p_0)$ is a point on $K$.
Assume further that, locally near $z_0$, the Lagrangian submanifold $K$ is a graph of a function $\set{p=f(q)},|f(q)|=1$.
The fact that $K$ is a Lagrangian submanifold then implies that $df$ is a symmetric matrix, and moreover, $T_{(q,p)}K\subseteq T_{(q,p)}T^*\R^d$ is the graph of $df$.
Consider the distribution of hyperplanes $\set{f(q)}^\perp \subseteq T_q \R^d$.
One can verify that when $df$ is symmetric then this distribution is (Frobenius) integrable.
Hence, there exists a neighborhood of $q_0$ that admits a local foliation by hypersurfaces which are tangent to this distribution.
It then also follows that for a given leaf $\Lambda$ of this distribution, $\set{(q,tf(q))\mid q\in\Lambda, t\in\R}$ is the normal bundle to $\Lambda$.
As a result, the restriction of $df$ to $\set{f(q)}^\perp$ is also the curvature operator of the leaf $\Lambda$ at $q$.
We think of the leaves of this foliation as the wave fronts of the beam of rays $q+\R f(q)$.
Suppose that $K_1$ and $K_2$ are two Lagrangian submanifolds of $T^*\R^d$ lying in the energy level $\set{|p|=1}$ that intersect at $z_0$, and locally they can be described as $K_1=\set{p=f(q)}$ and $K_2=\set{p=g(q)}$, where $|f(q)|=|g(q)|=1$.
Then the flow direction  $\R p_0$ is contained in the kernels of both $df$ and $dg$.
If $T_{z_0}K_1\cap T_{z_0}K_2 = \R p_0$, then it means that the restrictions of $df$ and $dg$ to $\set{p_0}^\perp$ do not agree on any non-trivial vector (if they agreed, then we would have a non-trivial vector in the intersection of their graphs, but the graphs of $df$ and $dg$ are $T_{z_0}K_1$ and $T_{z_0}K_2$, and their intersection is $\R p_0$).

If $K$ is not a graph near $z_0$, then we can let  the point $z_0$ flow for an arbitrary short time by the Hamiltonian flow of $H$, and near the resulting point, $K$ will be a graph.
This was proved in \cite{ArnoldV.I.1985TSta}.

The space $\mathcal{L}$ of oriented lines is obtained by symplectic reduction from $T^*\R^d$ with the Hamiltonian $H$. 
Take the level set $A=\set{H=\frac{1}{2}}=\set{|p|=1}$, and consider its quotient by the orbits of $H$.
The result is diffeomorphic to $\mathcal{L}$.
Suppose that $L\subseteq T_{\ell}\mathcal{L}$ is a Lagrangian subspace, then $L$ is a tangent space to a germ of a Lagrangian submanifold $Y$ passing through the point $\ell$.
Then, we can lift $Y$ to get a Lagrangian submanifold $K$ of $T^* \R^d$ which is invariant under the flow. 
Now we can apply the previous construction to cook foliations by wave fronts that will be orthogonal to the line $\ell$.
This foliation may fail to be defined at a finite set of points on $\ell$, where the submanifold $K$ fails to be a graph.

If we take two transversal Lagrangian subspaces $L_1, L_2$ of $T_{\ell}\mathcal{L}$, then 
the corresponding Lagrangian subspaces $T_{z_0}K_1$ and $T_{z_0}K_2$ in $T^*\R^d$ will intersect at $\R p_0$, where  $z_0=(q_0,p_0)$ and $\ell=q_0+\R p_0$. So by the previous remark, for any point $q_0\in\ell$
the curvature operators at $q_0$ of the corresponding fronts will not agree on any non-trivial vector, i.e., their difference is non-degenerate.

\section{Analysis of locally maximizing orbits}\label{section analysis}
In this section, we formulate a criterion for an orbit to be an m-orbit, and prove Theorem \ref{thm:criterionMinMax}. 
In Subsection \ref{subsection:criteriontori}, we provide a proof for twist maps of tori, and in Subsection \ref{subsection:criterioBallBundle} we formulate an analogous criterion for twist maps of ball bundles over spheres.
In Subsection \ref{subsection:examples}, we briefly mention that the results can be adapted to locally minimizing orbits, and we give two examples that relate to such orbits.
\subsection{Criterion for m-orbits}\label{subsection:criteriontori}
Let $T:T^*\T^n\to T^*\T^n$ be an exact twist map of $T^* \T^n$.
Denote by $(q,p)$ the symplectic coordinates on $T^*\T^n$, and denote by $H(q,Q)$ a generating function of $T$.
We assume that $H$ satisfies the twist condition, that is, the mixed partial derivatives matrix, $H_{12}$ is non-degenerate (see also Subsection \ref{subsection:twistmapintro}).
Denote by $\mathcal{M}$ the subset of $T^*\T^n$ that consists of all m-orbits, see Subsection \ref{subsubsection:morbits}.
Here are some important properties of this set:
\begin{proposition}\label{prop:elementaryProperties}
\begin{enumerate}\item\label{itm:lagrangianTori} The set $\mathcal{M}$ contains all globally maximizing orbits.
 If $T$ is a negative twist map, then $\mathcal{M}$ contains all orbits that lie on invariant Lagrangian graphs.
\item\label{itm:finiteSubsegment} If the matrix of the second variation of some finite segment of the configuration $\set{q_n}$ is negative semi-definite, then the matrix of the second variation for any proper sub-segment is negative definite.
In particular, a configuration $\set{q_n}$ is an m-orbit if and only if any finite segment has a negative semi-definite second variation.
\item\label{itm:closedInvariantSubset} The set $\mathcal{M}$ is a closed set, invariant under $T$.
\end{enumerate}
\end{proposition}
\begin{proof}
The first claim in item (\ref{itm:lagrangianTori}) is obvious. The fact that any orbit on a Lagrangian graph of a negative twist map is (globally) maximizing was proved by M. Herman \cite{MacKayRS1989CKtf,SEDP_1987-1988____A14_0}.
	 The claim in item \ref{itm:finiteSubsegment} was also proved in \cite{MacKayRS1989CKtf}.
	  It then follows from item (\ref{itm:finiteSubsegment}) that $\mathcal{M}$ is a closed set.
\end{proof}
\begin{definition}\label{def:JacobiField}
 A \textit{Jacobi field} along the configuration $\set{q_n}$ is a sequence of vectors $\set{\delta q_n}$ satisfying the discrete Jacobi equation:
\[b_{n-1}^T \delta q_{n-1}+a_n \delta q_n + b_n \delta q_{n+1}=0,\]
where, as in Subsection \ref{subsubsection:morbits}, the matrices $a_n$, $b_n$ are defined by:
\[a_n=H_{22}(q_{n-1},q_n)+H_{11}(q_n,q_{n+1}),b_n=H_{12}(q_n,q_{n+1}).\]
\end{definition}
In what follows, we consider the \textit{discrete matrix Jacobi equation} on the square matrices $\set{\xi_n}$, for $n\in\Z$:
\[b_{n-1}^T \xi_{n-1}+a_n \xi_n +b_n \xi_{n+1}=0.\]
We will need the following lemmas, which are based on ideas in \cite{BialyML2004Stmw}.
The first one can be seen as a higher dimensional analogue of the discrete Sturm comparison theorem \cite[Theorem 7.9]{ElaydiSaber1999Aitd}.
\begin{lemma}\label{lem:JacobiFieldLemma} Suppose that there exists a matrix Jacobi field $J_n$ of non-singular matrices along $\set{q_n}$, such that all the matrices $X_n:=-b_nJ_{n+1}J_n\inv$ are symmetric negative definite.
		Then, for any $k\in\Z$, the matrix Jacobi field $\xi^{(k)}_n$ satisfying the initial conditions $\xi^{(k)}_{k-1}=0,\  \xi^{(k)}_{k}=I $  is non-singular for all $n\geq k$.
		 Moreover, the matrices
		\[
		A^{(k)}_n:=-b_n\xi^{(k)}_{n+1}[\xi^{(k)}_n]^{-1},\]
	are all symmetric negative definite for $n\geq k$.
\end{lemma}
\begin{proof}
The sequence $\set{J_n}$ is a solution to the matrix Jacobi equation.
 Therefore, one can compute that the sequence $\set{X_n}$ satisfies the  recursion formula
\[X_{n+1}=a_{n+1}-b_n^TX_n^{-1}b_n.\]
Using the fact that $X_n$ are negative definite, we can conclude that
\begin{gather}\label{eq:anxn}
\forall n\in\Z,\quad X_n > a_n,
\end{gather}
	which means in particular that all $a_n$ are negative definite.
	
	For every $k$, consider the matrix Jacobi field $\set{\xi^{(k)}_n}_{n\in\Z}$ which satisfies the initial conditions 
	\[\xi^{(k)}_{k-1}=0,\  \xi^{(k)}_{k}=I.\]
	We need to show that $\xi^{(k)}_n$ are all non-singular for $n\geq k$.
	 We argue by contradiction.
	  Suppose for some $p\geq k$ the matrices $\xi^{(k)}_n$ are all non-singular when $n\in [k,p]$	while $\xi^{(k)}_{p+1}$ is singular.
	In this case, we can define for every $n\in[k,p]$
\[A^{(k)}_n:=-b_n\xi^{(k)}_{n+1}[\xi^{(k)}_n]\inv.\]
	Notice that for all $n\in[k,p-1]$, all matrices $A^{(k)}_n$ are non-singular, except the last one, 
	$A^{(k)}_{p}$ which is a singular matrix. 
	The recursion formula for $A^{(k)}_{n}$, which is the same as for $X_n$, is valid for $n\in[k,p-1]$:
\[A^{(k)}_{n+1}=a_{n+1}-b_n^T[ A^{(k)}_n]^{-1} b_n,\]
We easily compute from the definition that 
\[A^{(k)}_k=a_k,\]
 and from the recurrence relation we conclude that all matrices $A^{(k)}_n$ are symmetric for $n\in[k,p]$.
Hence, we get from (\ref{eq:anxn}):
\begin{gather*}
X_k > A_k ^{(k)}
\end{gather*}
Therefore, the monotonicity of the recurrence relation (see, e.g., \cite{MacKayRS1989CKtf}) implies that the inequality persists:
\begin{gather*}
X_{k+1}>A^{(k)}_{k+1}\ ,\dots\  X_{p-1}>A^{(k)}_{p-1},\ X_{p}>A^{(k)}_{p}.
\end{gather*}
This implies that all matrices $A^{(k)}_{n}, n\in[k,p]$ are negative definite, but this contradicts the fact that $A^{(k)}_{p}$ is singular.
This completes the proof of the lemma.
\end{proof}
\begin{lemma}\label{lem:BialyMackayGeneralization}
Suppose that $\set{q_n}$ is an m-configuration. 
Fix an integer $k$, and suppose that $\set{\xi_n^{({k})}}_{n\in\Z}$ is a matrix solution to the Jacobi equation with $\xi_{k-1}^{(k)}=0$, $\xi_{k}^{(k)}=I$.
Then $\xi_n^{(k)}$ is invertible for all $n\geq k$.
 \end{lemma}
 \begin{proof}
 For every vector $v$, the sequence $\{w_n:=\xi_n^{(k)}v, n\in\Z\}$ is a Jacobi field.
 Suppose that $v$ is a vector such that $\xi_p^{(k)}v=0$ for some $p>k$.
 Then we have $w_{k-1}=w_p=0,$ and additionally $ w_{k}=v$.
 Since $\set{q_n}$ is an m-configuration, then, in particular, the matrix of the second variation $W_{k,p-1}=\delta^2 F_{k,p-1}$ is negative definite.
 On the other hand,  the second variation computed on the variation
 $w=(w_{k},...,w_{p-1})$ obviously vanishes, because it follows from formula (\ref{eq:secondVariation}) that $W_{k,p-1}w^T=0$, since $w_n$ is a Jacobi field vanishing at $n=k-1,p$.
  Hence, $v$ must be $0$ and the matrix $\xi_n^{(k)}$ is indeed invertible.
 \end{proof}
We are now ready to prove Theorem \ref{thm:criterionMinMax}.
\begin{proof}[Proof of Theorem \ref{thm:criterionMinMax}]
The proof of the implication $\Longrightarrow$ can be found in \cite{BialyML2004Stmw} using Lemma \ref{lem:BialyMackayGeneralization}.
 For completeness, we will repeat the arguments here.

Suppose that $\set{q_n}$ is an m-orbit.
Fix an arbitrary $k\in \Z$.
Consider the matrix solution ${\xi_n^{(k)}, n\in\Z}$ to the Jacobi equation with the initial conditions $\xi_{k-1}^{(k)}=0$, $\xi_{k}^{(k)}=I$.
It follows from Lemma \ref{lem:BialyMackayGeneralization} that for all $n\geq k$ the matrices $A_n^{(k)}=-b_n\xi_{n+1}^{(k)}\Big(\xi_n^{(k)}\Big)\inv$ are well-defined. 
It is straightforward to verify that these matrices satisfy $$A_{k}^{(k)}=a_{k}\  {\rm and}\  A_{n+1}^{(k)}=a_{n+1}-b_n^T\Big(A_n^{(k)}\Big)\inv b_n, \forall n\geq k.$$
Since $a_n$ are symmetric, it follows that $A_n^{(k)}$ are also symmetric.

We need to show that these matrices are all negative definite. 
Take $\eta\in\R^d$ to be an arbitrary vector, and $m\geq k$ to be any integer.
For $n\geq k-1$, we set $\eta_n=\xi_n^{(k)}\Big(\xi_m^{(k)}\Big)\inv\eta$.
These vectors are obtained by multiplying the constant vector $\Big(\xi_m^{(k)}\Big)\inv \eta$ by a solution to the matrix Jacobi equation, so the result satisfies the vector Jacobi equation.
In addition, it holds that $\eta_{k-1}=0$, since $\xi_{k-1}^{(k)}=0$, and also, $\eta_m=\eta$.
Therefore, from formula  (\ref{eq:secondVariation}) we have, 
\[\delta^2 F_{k,m}(\eta_{k},...,\eta_{m})=-\inprod{b_m\cdot\eta_{m+1}}{\eta_m}=\inprod{A_m^{(k)}\eta}{\eta}.\]
So we conclude that
the negative definiteness of $A_m^{(k)}$ follows from that of $\delta^2 F_{k,m}$.

Next, we claim that for fixed $n$, the sequence $A_n^{(k)}$ is monotone in $k\leq n$.
To show this, fixing $k$, we prove by induction on $n\geq k$ that $A_n^{(k)} - A_n^{(k-1)}$ is negative definite.
For $n=k$, we have $A_{k}^{(k)}=a_{k}$ and $A_{k}^{(k-1)}=a_{k}-b_k^T\Big(A_{k-1}^{(k-1)}\Big)\inv b_k$, so then we get that $A_{k}^{(k)}-A_{k}^{(k-1)}=b_k^T\Big(A_{k-1}^{(k-1)}\Big)\inv b_k$, which is indeed negative definite.
For the induction step, if $A_n^{(k)}-A_n^{(k-1)}$ is negative definite, then:
\[A_{n+1}^{(k)}-A_{n+1}^{(k-1)}=-b_n^T\bigg(\Big(A_n^{(k)}\Big)\inv-\Big(A_n^{(k-1)}\Big)\inv\bigg)b_n,\]
and since $A_n^{(k)}$, $A_n^{(k-1)}$ and $A_n^{(k)}-A_n^{(k-1)}$ are also negative definite, then it follows that $A_{n+1}^{(k)}-A_{n+1}^{(k-1)}$ is also negative definite. 

As a result, for fixed $n$, the sequence $A_n^{(k)}$ is a sequence of negative  definite matrices, and it increases as $k$ decreases to $-\infty$, so it has a limit, $X_n$. 
The limit $X_n$ is negative semi-definite, but it must also be non-degenerate, since for all $k<n$ we have the recursive relation
\[A_{n+1}^{(k)}=a_{n+1}-b_n^T\Big(A_n^{(k)}\Big)\inv b_n,\]
which is equivalent to the following recursive relation that does not involve inverses of $A_n^{(k)},$
\[(A_{n+1}^{(k)}b_n\inv-a_{n+1}b_n\inv)A_n^{(k)}=-b_n^T.\]
Thus, the limit sequence $X_n$ also satisfies this relation, from which it follows that $X_n$ is non-degenerate (since so is $b_n$).
If we set $J_0$ to be an arbitrary invertible matrix, then we can define a sequence $\set{J_n}$ recursively by:
\begin{equation}\label{eq:JacobiFieldFromX}
\begin{cases}J_{n+1}=-b_n\inv X_n J_n,\quad n\geq 0,\\
J_n=-X_n\inv b_n J_{n+1},\quad n<0.\end{cases}\end{equation}
Then $\set{J_n}$ is a sequence of invertible matrices that solve the matrix Jacobi equation, since the recurrence relation holds for $X_n$, and $X_n=-b_n J_{n+1}J_n\inv$ is negative definite for all $n\in\Z$, as required.

Now we prove the converse statement.
Assuming that there exists a Jacobi field as in the formulation of the theorem, we prove, for example, that the matrix $W_{1N}$ is negative semi-definite (by using the notation of \eqref{eq:secondVariation}).
By item (\ref{itm:finiteSubsegment}) of Proposition \ref{prop:elementaryProperties}, this will show that $\set{q_n}$ is an m-orbit. 
We consider the matrices $A_n^{(k)}$ from Lemma \ref{lem:JacobiFieldLemma}.
The assumption of the theorem implies that the lemma can be used, so those matrices are all negative definite.
We evaluate the quadratic form $W_{1N}$ on a vector $u=(u_1,...,u_N)$ (here $u_i\in \R^n$ and $u\in\R^{nN}$) using the trick of \cite{MacKayRS1989CKtf} (this identity can be verified by induction on $N$):
\begin{gather*}
u^T W_{1N} u = \bigg(u_1 +\Big(A_1^{(1)}\Big)\inv b_1 u_2\bigg)^T A_1 ^{(1)} \bigg(u_1+\Big(A_1^{(1)}\Big)\inv b_1 u_2\bigg) + \\
+\bigg(u_2+\Big(A_2^{(1)}\Big)\inv b_2 u_3\bigg)^T A_2^{(1)}\bigg(u_2+\Big(A_2^{(1)}\Big)\inv b_2 u_3\bigg)+...+\\
+u_N^TA_N^{(1)}u_N.
\end{gather*}
Since all the matrices $A_n^{(1)}$ are negative definite, all sumands in this sum are negative, and hence, $W_{1N}$ is negative definite. 
Thus, $\set{q_n}$ is an m-orbit.
\end{proof}

\subsection{Adaptation to ball bundles over spheres}\label{subsection:criterioBallBundle}
Now we repeat the previous definitions and arguments, but for twist maps of a ball bundle $M\subseteq T^*N$, where $N\subseteq \R^d$ is a submanifold  diffeomorphic to $\Sph^{d-1}$,
for the definitions, see Subsection \ref{subsubsection:twistballbundle}, and Definition \ref{def:twistMapBallBundle}.
In this setting, we can also consider the discrete Jacobi equation (see Definition \ref{def:JacobiField}).
For a configuration $\set{q_n}$ of the map $T$ with generating function $H$, we can consider a Jacobi vector field $\set{\delta q_n}$, for $\delta q_n \in T_{q_n}N$ which satisfies the equation:
\[b_{n-1}^*\delta q_{n-1}+a_n\delta q_n+b_n\delta q_{n+1}=0,\]
where the operators $b_n=H_{12}(q_n,q_{n+1}), a_n=H_{11}(q_n,q_{n+1})+H_{22}(q_{n-1},q_n)$, and $b_{n-1}^* = H_{21}$ is the dual operator.

We will also consider \textit{an operator discrete Jacobi equation}: for a fixed $d-1$ dimensional vector space $E$, we consider maps $J_n:E\to T_{q_n}N$, and we look for maps that satisfy the equation
\[b_{n-1}^*\circ J_{n-1}+a_n\circ J_n + b_n\circ J_{n+1}=0.\]
With this definition, the proofs of Lemmas \ref{lem:JacobiFieldLemma}, \ref{lem:BialyMackayGeneralization} and Theorem \ref{thm:criterionMinMax} generalize to this case with some small adjustments:
\begin{enumerate}
\item\label{itm:xiknMap} In Lemma \ref{lem:JacobiFieldLemma} and Lemma \ref{lem:BialyMackayGeneralization}, one considers $\xi_n^{(k)}$ as a linear map from $T_{q_k}N\to T_{q_n}N$.
\item\label{itm:AnkXnmap} With this adjustment, the matrices $A_n^{(k)}$ and $X_n$ now correspond to linear endomorphisms of $T_{q_n}N$.
If we equip $N$ with the induced Riemannian metric from $\R^d$, then discussing the self-adjointness and negative definiteness of those operators is meaningful.
\item\label{itm:normalCoord}
Since all of our considerations are pointwise, we can always choose to work with normal local coordinates around each point $q_0$ of $N$.
In these coordinates, the Riemannian metric on $N$ at $q_0$ is the Euclidean metric.
These coordinates naturally induce coordinates $(q,p)$ on $M$ near the point $(q_0,p_0)$. 
It holds that $\set{(dq,dp)}$ are Darboux coordinates on $T_{(q_0,p_0)}M$.
\item\label{itm:nowallGood} From that point on, by representing all operators according to these coordinates, we get matrices, and the claims of the previous section follow through to this case, verbatim. 
\end{enumerate}
Thus, we can state the following theorem, which is a generalization of Theorem \ref{thm:criterionMinMax}, for both: twist maps of cotangent bundles of a torus, and twist maps of ball bundles over spheres, and work in a unified way.
\begin{theorem}\label{prop:mOrbitCriterioBallBundle}
Suppose that $N$ is either $\T^{d-1}$ or diffeomorphic to $\mathbb{S}^{d-1}$, and $M$ is either $T^* \T^{d-1}$ or a ball bundle over $N$.
Suppose that $T$ is an exact twist map of $M$ with a generating function $H$.
Then a sequence $\set{(q_n,p_n)}$ is an m-orbit of $T$ if and only if there exists an operator Jacobi field $J_n: E\to T_{q_n} N$ of non-degenerate operators, for which $X_n = -b_n J_{n+1} J_n\inv$ are all negative definite endomorphisms of $T_{q_n}N$ (here $E$ is an arbitrary $d-1$ dimensional vector space).
\end{theorem}
\subsection{Minimizing Orbits}\label{subsection:examples}
If one inspects carefully the proofs from Subsection \ref{subsection:criteriontori}, then it can be seen that a version of Theorem \ref{thm:criterionMinMax} can be derived for locally minimizing orbits.
In this case, the condition is the existence of a Jacobi field $J_n$, for which the symmetric matrix $X_n$ will be positive definite. 
In the case of positive twist maps of a two-dimensional cylinder, this results in the existence of Jacobi field of constant sign, as was discussed in \cite{Bialy1993, Michael2012Hrfc, MacKayR.S1985CKTa, bialy2022locally}.
However, in the case of a negative twist maps in two dimensions, this results in a Jacobi field with alternating signs.
Here are two examples of this phenomenon.
\subsubsection{Standard twist map}\label{itm:standardMapExample} Consider an exact symplectic map of the cylinder with the generating function
\begin{equation*}
H(q,Q)=-\left(\frac{1}{2}(q-Q)^2+V(q)\right),\quad 
\begin{cases}
	P=p-V'(q)\\
	Q=q-p+V'(q)
		\end{cases},
	\end{equation*}
where $V$ is $2\pi$-periodic function. 
Notice that by replacing $p$ with $-p$ we get the usual standard-like map.
Suppose that $V$ has local maxima at the points $2\pi i/N$, $i=1,\dots,N$, and $N>1$ is some positive integer.
 Then the mapping $T$ has an $N$-periodic orbit: $(q=0, p=2\pi/N)$. 
 The orbit of this point has the matrix of the second variation \eqref{eq:secondVariation} with the entries
\[b_i=-1, a_i=-2-V''(q_i).\]
Therefore, if the maxima points of $V$ are steep enough, namely $V''(\frac{2\pi i}{N})<-4$, then this orbit is locally minimizing between any two points.
Since the mixed partial derivative of $H$ is negative, this means that the corresponding Jacobi field will have alternating signs.
\subsubsection{Billiards in polygons}\label{itm:billiardPolygonExample}
Consider a billiard motion in any convex polygon $\gamma$. 
Then any orbit is locally minimizing. 
Indeed, this orbit has  a Jacobi field with an alternating sign.
 This field can be constructed as follows.
  Take the coordinates $(s, \delta)$, where $s$ is the arc length parameter on the boundary of the table, and $\delta$ is the angle from the tangent.
  Consider an orbit $(s_n,\delta_n)$, and a tangent vector $\frac{\partial}{\partial s}\in T_{\gamma(s_0)}\gamma$.
Then this vector is transformed by the billiard map to a vector which is of the form $\lambda\frac{\partial}{\partial s}$, where $\lambda$ is a negative factor. 
Iterating, one gets alternating Jacobi field, see Figure \ref{fig:minimizingOrbitPolygon}.
\begin{figure}
	\centering
	\begin{tikzpicture}[scale = 1.5]
	\begin{scope}[decoration={
markings,
mark = at position 0.7 with {\arrow{>}}}
]
\tkzDefPoint(1,0){A};
\tkzDefPoint(-1,0){B};
\tkzDefPoint(0,3){C};
\tkzDrawPoints[blue](A,B,C);
\tkzDrawSegment[blue](A,B);
\tkzDrawSegment[blue](B,C);
\tkzDrawSegment[blue](C,A);
\tkzDefPoint(0,0){D};
\tkzDrawPoint(D)
\draw[red, line width = 2pt,  postaction = {decorate}](0.5,0)--(-0.5,0);
\tkzDefPoint(-0.5,1.5){E};
\tkzDrawPoint(E)
\draw[black, postaction = {decorate}](0,0)--(-0.5,1.5);
\draw[red,dashed, postaction = {decorate}](-0.5,0)--(-0.75,0.75);
\draw[red,dashed, postaction = {decorate}](0.5,0)--(-0.25,2.25);
\draw[red, line width = 2pt,  postaction = {decorate}](-0.25,2.25)--(-0.75,0.75);

\end{scope}
	\end{tikzpicture}
	\caption{If we consider a family of orbits with the same angle (the red dashed segments), then the orientation reverses after reflection.
	 This means that the Jacobi field has alternating signs. \label{fig:minimizingOrbitPolygon}}
\end{figure}
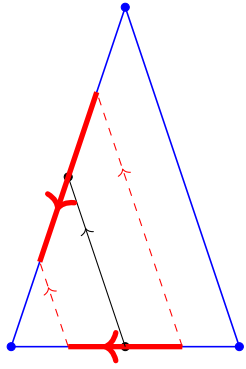
Of course, one can smoothen the angles of the polygon to get a smooth convex billiard table.
It is interesting if one can improve this construction in order to get a billiard table with strictly positive curvature that has a minimizing orbit.
\section{Geometric meaning of the criterion}\label{section:geometricMeaning}
\subsection{Geometric criterion}\label{subsection:criterionGeom}
To describe the geometric meaning of the criterion of Theorem \ref{thm:criterionMinMax} proved in Section \ref{section analysis}, we will use the partial order relation that is defined on Lagrangian subspaces, see Subsection \ref{subsection:LagrangianSubspaces}.

Suppose that $N$ is either $\mathbb T^{d-1}$ or a submanifold of $\R^d$ diffeomorphic to $\Sph^{d-1}$, and $M$ is either $T^*\mathbb T^{d-1}$ or a ball bundle over $N$, respectively.

Let $T:M\to M$ be a twist map with a generating function $H$.
Suppose that $\set{(q_n,p_n)}$ is an arbitrary orbit of $T$.
Let $\pi:M\to N$ be the canonical projection, and denote by $V_n = \ker d\pi_{(q_n,p_n)}$ the vertical subspace at the point $(q_n,p_n)$, and set
$\alpha_n=dT(V_{n-1})$, $\beta_n = dT\inv(V_{n+1})$.
Then $V_n$, $\alpha_n$, $\beta_n$ are all Lagrangian subspaces of $T_{(q_n,p_n)}M$, and $\alpha_n$, $\beta_n$ are transversal to $V_n$ by the twist condition.
We start with the following lemma.
Recall that we are working with the coordinates induced on $M$ by the normal local coordinates of $N$ around each point $q_n$ of the orbit $\set{(q_n,p_n)}$, see item \ref{itm:normalCoord} in Subsection \ref{subsection:criterioBallBundle}.
\begin{lemma}\label{lem:recursionAndInvariance}
Let $T$ be a twist map with a generating function $H$, and let $\set{(q_n,p_n)}$ be any orbit of $T$.
Consider a sequence $X_n$ of self-adjoint automorphisms of $T_{q_n}N$.
Define a Lagrangian subspace $L_n$ to be the graph of the endomorphism 
\[W_n=-H_{11}(q_n,q_{n+1})+X_n,\]
i.e., the set $\set{dp=W_n dq}\subseteq T_{(q_n,p_n)}M$, in the $(q,p)$ coordinates.
Then $L_n$ is $T$-invariant if and only if $X_n$ satisfies the recursion formula
\[X_{n+1} = a_{n+1}-b_n^* X_n\inv b_n.\]
\end{lemma}
\begin{proof}
First, we give an explicit description for the differential of $T$.
We can write $T=\tilde{T}\circ \varphi$, where $\varphi(q,p)=(q,Q)$ is a diffeomorphism by the twist condition (see item \ref{itm:twistDiffeo} of Definition \ref{def:twistMapDef}), and $\tilde{T}(q,Q)=(Q,P)$.
The inverse of $\varphi$ and $\tilde{T}$ can be described with the generating function $H$:
\begin{gather*}
\varphi\inv (q,Q)=(q,-H_1(q,Q)),\\
\tilde{T}(q,Q)=(Q,H_2(q,Q)).
\end{gather*}
Hence the differentials have the following block representation in normal coordinates:
\begin{gather*}
d\varphi\inv = \begin{pmatrix}
I & 0 \\
-H_{11} & -H_{12} 
\end{pmatrix},\quad
d\tilde{T}=\begin{pmatrix}
0 & I \\
H_{21} & H_{22}
\end{pmatrix},
\end{gather*}
and also, $d\varphi = \begin{pmatrix}
I & 0 \\
-H_{12}\inv H_{11} & -H_{12}\inv
\end{pmatrix}$.
Therefore, we have the following description of the $dT$:
\[dT=d\tilde{T}d\varphi = \begin{pmatrix}
-H_{12}\inv H_{11} &  -H_{12}\inv \\
H_{21} - H_{22}H_{12}\inv H_{11} & -H_{22}H_{12}\inv
\end{pmatrix}.\]
Next, we find the image of $L_n$ by $dT$.
 Since
$L_n$ is the graph of $W_n$, then we need to compute
\begin{gather*}
dT(L_n)=\set{\begin{pmatrix}
-H_{12}\inv H_{11} &  -H_{12}\inv \\
H_{21} - H_{22}H_{12}\inv H_{11} & -H_{22}H_{12}\inv
\end{pmatrix}
\begin{pmatrix}
v \\
(-H_{11}+X_n)v
\end{pmatrix}} = \\
=\set{
\begin{pmatrix}
-H_{12}\inv X_n v \\
(H_{21}-H_{22}H_{12}\inv X_n)v
\end{pmatrix}}
\end{gather*}
We need to compare the result to $L_{n+1}$.
Therefore, we check that the equality
\[W_{n+1}(-H_{12}\inv X_n v) = (H_{21}-H_{22}H_{12}\inv X_n)v\]
holds if and only if $X_n$ satisfies the recursion formula.
The former is equivalent to checking that $-W_{n+1}H_{12}\inv X_n = H_{21}-H_{22}H_{12}\inv X_n$, which is equivalent to $W_{n+1}=H_{22}-H_{21}X_n\inv H_{12}$.
And indeed, this equation holds if and only if $-H_{11}(q_{n+1},q_{n+2})+X_{n+1}=H_{22}(q_n,q_{n+1})-b_n^* X_n\inv b_n$, which is equivalent to $X_{n+1}=H_{22}+H_{11}-b_n^* X_n\inv b_n$, and this is the required recursive relation.
\end{proof}
Before proving Theorem \ref{thm:morbitInequality}, we make the following remarks.
\begin{remark*}

\begin{enumerate}
\item Theorem \ref{thm:morbitInequality} implies, in particular, that at the points of a locally maximizing orbit we have $\alpha_n \subspaceleq{V_n}\beta_n$.
\item\label{itm:oneIneqIsEnough} As the proof of the theorem shows, in the converse direction, it is enough to assume that $L_n$ satisfies the assumptions \ref{itm:transversality}, \ref{itm:invariance}, and that either the left inequality of \ref{itm:inequality} holds for all $n\in\Z$, or that the right one holds for all $n\in\Z$. 
\item Since Theorem \ref{prop:mOrbitCriterioBallBundle} works also for locally minimizing orbits (see Subsection \ref{subsection:examples}), then this theorem can also be applied for locally minimizing orbits.
In this case, the inequalities in item \ref{itm:inequality} will be reversed.
\end{enumerate}
\end{remark*}

\begin{proof}[Proof of Theorem \ref{thm:morbitInequality}]
Write $H$ for the generating function of $T$.
Recall that $\alpha_n$ and $\beta_n$ are the graphs of the self-adjoint endomorphisms $H_{22}(q_{n-1},q_n)$, $-H_{11}(q_n,q_{n+1})$, respectively. 
First, assume that $\set{(q_n,p_n)}$ is an m-configuration.
Then, given by Theorem \ref{prop:mOrbitCriterioBallBundle}, there exists a Jacobi field $J_n$ for which $X_n=-b_n J_{n+1}J_n\inv$ is negative definite, and in particular, an automorphism.
Define the subspace $L_n$ to be the graph of the endomorphism
\[W_n = -H_{11}(q_n,q_{n+1})+X_n.\]
The endomorphism $W_n$ is self-adjoint, so $L_n$ is indeed a Lagrangian subspace. 
Using the fact that $J_n$ is a Jacobi field, one can compute that $X_n$ satisfies the recursive relation
\[X_{n+1}=a_{n+1}-b_n^* X_n\inv b_n.\]
Hence, Lemma \ref{lem:recursionAndInvariance} implies that $L_n$ is an invariant subspace field, so item \ref{itm:invariance} holds.
Item \ref{itm:transversality} follows since $L_n$ is a graph of some operator.
If $\set{(q_n,p_n)}$ is an m-orbit, then $X_n$ is negative definite, and therefore $W_n<-H_{11}$, so we have $L_n\subspaceleq{V_n}\beta_n$.
Also, we can write, thanks to the recursive relation of $X_n$,
\[W_n=-H_{11}+a_n-b_{n-1}^*X_{n-1}\inv b_{n-1}=H_{22}(q_{n-1},q_n)-b_{n-1}^* X_{n-1}\inv b_{n-1},\]
and as a result $H_{22} < W_n$, so we also have $\alpha_n\subspaceleq{V_n} L_n$.
This proves item \ref{itm:inequality}.

Now we prove the converse statement.
Suppose $L_n$ is Lagrangian subspace field, that satisfies items \ref{itm:transversality}, \ref{itm:invariance}, and one of the inequalities in item \ref{itm:inequality}.
Let $W_n$ be the self-adjoint operator for which $L_n$ is the graph.
If the inequality in item \ref{itm:inequality} is  $\alpha_n \subspaceleq{V_n}L_n $, then we can define $X_n$ by 
\[X_n = W_n+H_{11}(q_n,q_{n+1}),\]
 and if the inequality is $\beta_n \subspaceleq{V_n} L_n$, then we can define $X_n$ by 
 \[X_n = b_{n-1}\inv (H_{22}(q_{n-1},q_n)-W_n)\inv (b_{n-1}\inv)^*.\]
The inequality in item \ref{itm:inequality} implies that $X_n$ is negative definite, and in particular, an automorphism.
The fact that $L_n$ is invariant implies by Lemma \ref{lem:recursionAndInvariance} that $X_n$ satisfies the recursive relation written there (technically, the lemma addresses only $X_n$ that are defined in terms of $H_{11}$. 
If $X_n$ is defined in terms of $H_{22}$, then the proof is analogous).
Since $X_n$ satisfies the recursive relation, we can define a Jacobi field by the recipe of \eqref{eq:JacobiFieldFromX}, and then Theorem \ref{prop:mOrbitCriterioBallBundle} will imply that $\set{(q_n,p_n)}$ is an m-configuration.
The proof is completed.

\end{proof}

\subsection{Two generating functions}
In this subsection we recall the Geometric Assumption \ref{GeometricAssumption} from Section \ref{section:intro}, and show that it guarantees that two generating functions of the same twist map will have the same sets of m-orbits.
Suppose now that the twist map $T$ has two generating functions, $H^{(1)}$ and $H^{(2)}$, with respect to two sets of symplectic coordinates.
 Recall that we denote by $\mathcal{M}_{H^{(1)}}$, $\mathcal{M}_{H^{(2)}}$ the  sets that consists of all m-orbits for the generating functions $H^{(1)}$, $H^{(2)}$, respectively.

We now restate the Geometric Assumption \ref{GeometricAssumption}, and show that it ensures that the m-orbits according to the two generating functions coincide.
We denote the vertical bundles for each set of coordinates by $V^{H^{(1)}}$ and $V^{H^{(2)}}$, and write 
\begin{gather*}
\alpha^{H^{(1)}}=dT(V^{H^{(1)}}),\,, \alpha^{H^{(2)}} = dT(V^{H^{(2)}}), \\
   \beta^{H^{(1)}}=dT\inv(V^{H^{(1)}}),\,,
  \beta^{H^{(2)}}=dT\inv(V^{H\powbra{2}}).
   \end{gather*}

\textbf{Geometric Assumption}
\begin{itemize}
\myitem{(GA)}\label{GeometricAssumption} Suppose that at every point of $\mathcal{M}_{H\powbra{1}}\cup\mathcal{M}_{H\powbra{2}}$ 
 there exists a homotopy of Lagrangian subspaces $\set{V_t}$ connecting $V^{H\powbra{1}}$ and $V^{H\powbra{2}}$, such that for all $t$, the subspace $V_t$ is transversal to all four subspaces $\alpha^{H\powbra{1}}$, $\beta^{H\powbra{1}}$, $\alpha^{H\powbra{2}}$, $\beta ^{H\powbra{2}}$.

\end{itemize}
Note that we only require this homotopy to exist pointwise, and we do not assume that it needs to be uniform in some way.
In the two-dimensional case, in \cite{bialy2022locally}, we used a condition that is somewhat similar to the condition \ref{GeometricAssumption}.

Now we turn to prove Theorem \ref{thm:geometricAssumptionImpliesMOrbitEq}.

\begin{proof}[Proof of Theorem \ref{thm:geometricAssumptionImpliesMOrbitEq}]
Suppose that $\set{(q_n,p_n)}$ is an m-orbit with respect to $H\powbra{1}$.
From Theorem \ref{thm:morbitInequality}, there exists an invariant field of Lagrangian subspaces $L$ along the orbit that is transversal to $V^{H\powbra{1}}$, and for which 
\[ \alpha^{H\powbra{1}}\subspaceleq{V^{H\powbra{1}}} L \subspaceleq{V^{H\powbra{1}}} \beta^{H\powbra{1}}.\]
In particular, $ \alpha^{H\powbra{1}}\subspaceleq{V^{H\powbra{1}}} \beta^{H\powbra{1}}$, and hence by Proposition \ref{prop:orderQuadraticForm}
\[Q[\alpha^{H\powbra{1}},\beta^{H\powbra{1}}]\mid_{V^{H\powbra{1}}} > 0.\]
Now we use Lemma \ref{lem:homotopyTransversalityPersistence} to see that for all $t\in[0,1]$
\[Q[\alpha^{H\powbra{1}},\beta^{H\powbra{1}}]\mid_{V_t} > 0.\]
On the other hand, the inequality $\alpha^{H\powbra{1}} \subspaceleq{V_{H\powbra{1}}} L \subspaceleq{V_{H\powbra{1}}} \beta^{H\powbra{1}}$ means, according to item \ref{itm:doubleInequalityNegative} of Proposition \ref{prop:orderQuadraticForm}, that 
\[Q[\alpha^{H\powbra{1}},\beta^{H\powbra{1}}]\mid_L < 0.\]
Therefore, that $L$ is necessarily transversal to all the subspaces $V_t$.
In particular, $L$ is transversal to $V^{H\powbra{2}}$, and hence also to $\alpha^{H\powbra{2}}$ and $\beta^{H\powbra{2}}$.
Thus, $L$ satisfies the first two items of Theorem \ref{thm:morbitInequality}.
To conclude that $\set{(q_n,p_n)}$ is an m-orbit for $H\powbra{2}$, it is enough to show that one of the inequalities mentioned in item \ref{itm:inequality} in that theorem holds (see remark \ref{itm:oneIneqIsEnough} after the Theorem).
So we prove the inequality $  L\subspaceleq{V^{H\powbra{2}}}\beta^{H\powbra{2}}$.
We start from the inequality $\alpha^{H\powbra{1}} \subspaceleq{V^{H\powbra{1}}}L$.
By Proposition \ref{prop:orderQuadraticForm} we have
\[Q[\alpha^{H\powbra{1}},L]\mid_{V^{H\powbra{1}}}>0.\]
Since the subspaces $V_t$ are transversal to both $L$ and $\alpha^{H\powbra{1}}$, we can use Lemma \ref{lem:homotopyTransversalityPersistence} to claim that for all $t\in[0,1]$
\[Q[\alpha^{H\powbra{1}},L]\mid_{V_t}>0,\]
and in particular, $Q[\alpha^{H\powbra{1}},L]\mid_{V^{H\powbra{2}}}$ is positive definite.
This claim is invariant under the action of a symplectomorphism, so we can apply $T\inv$ to all three subspaces, and see that $Q[V^{H\powbra{1}},L]\mid_{\beta^{H\powbra{2}}}$ is also positive definite.
Now, we shift cyclically \cite[Remark 2]{ArnoldV.I.1985TSta} to get that
\[Q[L,\beta^{H\powbra{2}}]\mid_{V^{H\powbra{1}}} > 0.\]
Next, we use Lemma \ref{lem:homotopyTransversalityPersistence} to conclude that this inequality persists through the homotopy, and hence
\[Q[L,\beta^{H\powbra{2}}]\mid_{V^{H\powbra{2}}}>0.\]
By Proposition \ref{prop:orderQuadraticForm}, we can finally conclude that $L\subspaceleq{V^{H\powbra{2}}} \beta^{H\powbra{2}}$, which is the desired inequality.
\end{proof}

\begin{remark*}
The same exact proof can be used to show that the geometric assumption \ref{GeometricAssumption} also implies the equality of the sets of minimizing orbits for both generating functions.
\end{remark*}
\section{Application to multi-dimensional billiards}\label{section:highDimBilliard}
Our goal in this section is to prove Theorem \ref{thm:equalityForBilliard}: to show that the two generating functions for the billiard dynamical system, mentioned in Subsection \ref{subsection:generatingFunctionsBilliards}, have the same m-orbits.
To that end, we show that they satisfy the Geometric Assumption \ref{GeometricAssumption}, and hence Theorem \ref{thm:equalityForBilliard} follows from Theorem \ref{thm:geometricAssumptionImpliesMOrbitEq}.
At this point, it is not clear if the minimizing orbits for these two generating functions are the same.

We use the machinery of wave fronts, see Subsection \ref{subsection:wavefront}.
This allows us to think of local foliations by hypersurfaces (wave fronts) in $\R^d$ instead of Lagrangian subspaces of the tangent space of $\mathcal{L}$, the space of oriented lines.
So instead of describing a homotopy of Lagrangian subspaces, we can describe a homotopy of hypersurfaces.
As explained in Subsection \ref{subsection:wavefront}, the condition that two subspaces will be transversal is equivalent, in this language, to the condition that the difference between the two corresponding hypersurfaces' curvature operators will be non-degenerate.
First, we recall how the hypersurface that represents such a Lagrangian subspace evolves under the billiard map.
\subsection{Sinai-Chernov formula}\label{subsection:sinai}
Consider an oriented line $\ell\in\mathcal{L}$ and a Lagrangian subspace $Y\subseteq T_\ell\mathcal{L}$.
As we explained, there exists a local foliation by hypersurfaces that corresponds to $Y$, for which $Y$ is the tangent space at $\ell$ of the normal bundle of those hypersurfaces.
Every two leaves of the foliation are related by \textit{free flight}, and there is a simple connection between their curvature operators, which is a manifestation of Huygens principle:
\begin{proposition}\label{prop:freeFlight}
Let $\ell$ be an oriented line, and $Y\subseteq T_\ell\mathcal{L}$ be a Lagrangian subspace.
Suppose that $\Sigma_1$ and $\Sigma_2$ are two leaves of the foliation determined by  $Y$.
Write $q_i$ for the intersection point of $\Sigma_i$ and $\ell$, and write $B_i$ for the curvature operator of $\Sigma_i$ at $q_i$.
If $r$ is the (signed) distance between $q_1$ and $q_2$ (computed according to the orientation of $\ell$), then

\begin{equation}\label{eq:B1B2}
B_2(rB_1+I)=B_1.
\end{equation}
Equation \eqref{eq:B1B2} is understood with the help of the identification of the tangent spaces of the fronts $T_{q_1}\Sigma_1$ and $T_{q_2}\Sigma_2$.
\end{proposition}
Now we describe how the curvature operator is affected by a billiard collision.
This is covered by the well-known Sinai-Chernov formula (see \cite{SinaiChernovFormula,Chernov06chaoticbilliards,AST_2003__286__119_0}) .

\begin{proposition}\label{prop:billiardCollision}
Let $\Sigma\subseteq \R^d$ be a smooth convex hypersurface.
Suppose that at a point $y\in \Sigma$ a billiard collision occurs, from an incoming ray $\ell_1$ with direction $u_1$ to an outgoing ray $\ell_2$ with direction $u_2$.
Suppose that $Y\subseteq T_{\ell_1}\mathcal{L}$ is a Lagrangian subspace.
Write $\Sigma_1$ for the leaf of the foliation defined by $Y$ (through the method of Subsection \ref{subsection:wavefront}) that passes through $y$, and $B_1$ for its curvature operator at $y$.
Write $T$ for the billiard map in $\Sigma$. 
Then $dT_{\ell_1}(Y)\subseteq T_{\ell_2}\mathcal{L}$ is another Lagrangian subspace.
Suppose that $\Sigma_2$ is the leaf of the foliation that this subspace defines, which passes through $y$, and write $B_2$ for its curvature operator at $y$.
Then
\[B_2 = R^* B_1 R + 2\inprod{u_2}{n_y}(p_{n_y,u_2}\inv)^* dG_y p_{n_y,u_2}\inv,\]
where $n_y$ is the outer unit normal to $\Sigma$ at $y$,
 $dG_y$ is the differential of the Gauss map at $y$ (i.e., the curvature operator of $\Sigma$ at $y$),
  $R$ is the restriction of the orthogonal reflection about $\set{n_y}^\perp$ to a map from $\set{u_2}^\perp$ to $\set{u_1}^\perp$,
   and for two unit vectors $a$, $b$, $p_{a,b}$ denotes the restriction of the orthogonal projection on $\set{b}^\perp$ to $\set{a}^\perp$.

\end{proposition}
For completeness, we provide a proof in Appendix \ref{app:sinaiChernovProof}.
\subsection{Description of homotopy}\label{subsection:homotopy}
In this subsection, we describe a homotopy $V_s$, $s\geq 0$, between the subspaces $V^S$ and $V^L$ of $T_{\ell}\mathcal{L}$ for some oriented line $\ell$.
The geometric idea is very simple and is demonstrated in Figure \ref{fig:homotopy}.
  \begin{figure}
        \centering
        \begin{tikzpicture}[scale = 1.75]
        	\begin{scope}[decoration={
markings,
mark = at position 1 with {\arrow{>}}}
]
\draw[black, line width = 1pt, postaction = {decorate}](1,0)--(0.5,0);
\end{scope}
\tkzDefPoint(1,0){A};
\tkzDrawPoint[size = 4pt](A);
\draw[dashed](1,0)--(4,0);
\tkzDefPoint(2,0){B};
\tkzDrawPoint[size=4pt,blue](B);
\tkzDefPoint(3,0){C};
\tkzDrawPoints[size=4pt,red](C);
\draw[domain = -10:100, smooth, variable = \t, black, line width = 2.5pt] plot({0+1.414*cos(\t)},{-1+1.414*sin(\t)});

\draw[domain = 150:210, smooth, variable = \t, blue, line width = 1pt,dashed] plot({cos(\t)+2},{sin(\t)});

\draw[domain = 165:195, smooth, variable = \t, red, line width = 1pt,dashed] plot({2*cos(\t)+3},{2*sin(\t)});

\node[black] at (1.1,-0.8) {$\Sigma$};
\node[black] at (0.9,0.2) {$y$};
\node[blue] at (2,-0.3) {$y-s_1 u$};
\node[red] at (3,-0.3) {$y-s_2 u$};
\node[blue] at (1.5,0.2) {$B_{s_1}=\frac{1}{s_1}I$};
\node[red] at (0.7,0.8) {$B_{s_2}=\frac{1}{s_2}I$};
\node[black] at (0.4,0) {$u$};
       \end{tikzpicture}
       \caption{Homotopy of wave fronts and subspaces.\label{fig:homotopy}}
    \end{figure}
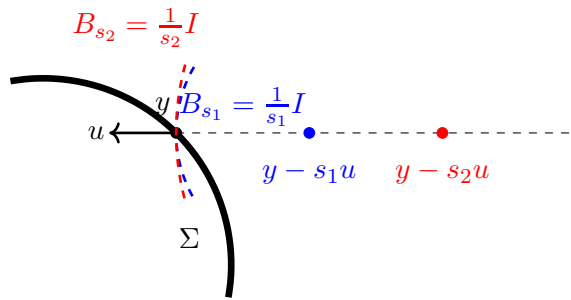
Write $y$ for the intersection point of $\ell$ with $\Sigma$, and $u$ for its unit direction.
The vertical subspace of the $L$ coordinates at $T_{\ell}\mathcal{L}$ corresponds to rays emanating from the point $y$. 
These rays induce a local foliation by spheres centered at $y$.
The vertical subspace of the $S$ coordinates at $T_{\ell}\mathcal{L}$ corresponds to rays that are parallel to $\ell$.
The corresponding local foliation is induced by hyperplanes orthogonal to $\ell$.
Thus, at the level of hypersurfaces, we can define the homotopy as the homotopy that moves the center of the sphere from $y$ to $y-su_2$, where $s\in(0,\infty)$.
The leaf of the foliation that passes through $y$ is then a sphere of radius $s$ centered at $y-su_2$, and the curvature operator of this hypersurface is $B_s = \frac{1}{s}I$.
In terms of Lagrangian subspaces, the subspace $V_s$, for $s\in (0,\infty)$, is just the tangent space at $\ell$ of the normal bundle to the sphere of radius $s$ centered at $y-su$.
This extends continuously when $s\to 0$ to the subspace $V^L$, and when $s\to\infty$ to the subspace $V^S$.
As a result, this homotopy, as a homotopy of the Lagrangian subspaces, is indeed continuous for $s\in[0,\infty]$.

\subsection{Verification of the criterion}
\label{subsection:verification}
Now we are ready to prove Theorem \ref{thm:equalityForBilliard}.
We do it by verifying that the hypotheses of Theorem \ref{thm:geometricAssumptionImpliesMOrbitEq} are satisfied by the m-orbits of the billiard map. 
Thus, we verify the Geometric Assumption \ref{GeometricAssumption} for the generating functions $L$, $S$.
Fix an oriented line $\ell_2 = y+\R u_2\in\mathcal{L}$, $y\in\Sigma$, $|u_2|=1$.
We use the homotopy described in Subsection \ref{subsection:homotopy}.
Our goal is to show that whenever we have an m-orbit for one of the coordinates, then the difference between $B_s$ and the curvature operator of the front corresponding to $\alpha^S$, $\beta^S$, $\alpha^L$, $\beta^L$, denoted by $B$, is non-degenerate (as was explained in Subsection \ref{subsection:wavefront}, this is equivalent to the subspaces being transversal).
Since $B_s=\frac{1}{s}I$, for $s> 0$ (including $\infty$), this is equivalent to $B$ having no positive eigenvalues, i.e., that $B$ is negative definite.

The condition that the point is an m-orbit implies, by Theorem \ref{thm:morbitInequality}, that $\alpha^S\subspaceleq{V^S} \beta^S$, or that $\alpha^L \subspaceleq{V^L} \beta^L$, and this means that $S_{11}+S_{22}<0$, or $L_{11}+L_{22}<0$ for all points of the orbit.  
The verification will be done at the tangent space to the line $\ell_2 = y+\R u_2$.
First, let us fix the notation.
Suppose the billiard orbit is the following:
\[...\rightarrow a \xrightarrow {u_0} x \xrightarrow{u_1} y \xrightarrow{u_2} z \xrightarrow{u_3} b \rightarrow...\]
Namely, the collision points on $\Sigma$ are $a$, $x$, $y$, $z$, $b$, in this order, and the vectors $u_i$ are the unit vectors between two subsequent points, see Figure \ref{fig: fivePointCollision}.
  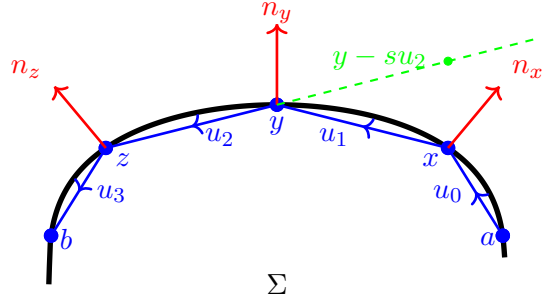
\begin{figure}
        \centering
       \begin{tikzpicture}[scale = 1.5]
   \draw[domain = -2:2, smooth, variable = \t, black, samples = 300, line width = 2pt] plot({\t},{(4-(\t)^2)^(1/3)});
\tkzDefPoint(0,1.587){A};
\tkzDefPoint(1.5,1.205){B};
\tkzDefPoint(-1.5,1.205){C};
\tkzDefPoint(1.98,0.43){D};
\tkzDefPoint(-1.98,0.43){E};
\tkzDrawPoints[size = 5pt,blue](A,B,C,D,E);
	\begin{scope}[decoration={
markings,
mark = at position 0.5 with {\arrow{>}}}
]
\draw[blue, line width = 1pt, postaction = {decorate}](1.98,0.43)--(1.5,1.205);
\draw[blue, line width = 1pt, postaction = {decorate}](1.5,1.205)--(0,1.587);
\draw[blue, line width = 1pt, postaction = {decorate}](0,1.587)--(-1.5,1.205);
\draw[blue, line width = 1pt, postaction = {decorate}](-1.5,1.205)--(-1.98,0.43);

\end{scope}
\tkzDrawPoints[size = 5pt,blue](A,B,C,D,E);
	\begin{scope}[decoration={
markings,
mark = at position 1 with {\arrow{>}}}
]
\draw[red, line width = 1pt, postaction = {decorate}](0,1.587)--(0,2.3);
\draw[red, line width = 1pt, postaction = {decorate}](1.5,1.205)--(1.5+0.15*2*1.5,1.205+0.15*3*1.205);
\draw[red, line width = 1pt, postaction = {decorate}](-1.5,1.205)--(-1.5-0.15*2*1.5,1.205+0.15*3*1.205);

\end{scope}
\draw[dashed,green,line width = 0.8pt](0,1.587)--(0+1*1.5*1.5,1.587+1.5*0.385);
\tkzDefPoint(1*1.5,1.587+1*0.385){F}
\tkzDrawPoint[size = 3pt, green](F);
\node at (0,0) {$\Sigma$};
\node[blue] at (0,1.4) {$y$};
\node[blue] at (1.35,1.1) {$x$};
\node[blue] at (-1.35,1.1) {$z$};
\node[blue] at (1.85,0.4) {$a$};
\node[blue] at (-1.85,0.4) {$b$};

\node[blue] at (1.5,0.8) {$u_0$};
\node[blue] at (-1.45,0.8) {$u_3$};
\node[blue] at (0.5,1.3) {$u_1$};
\node[blue] at (-0.5,1.3) {$u_2$};

\node[red] at (0,2.4) {$n_y$};
\node[red] at (2.2,1.9) {$n_x$};
\node[red] at (-2.2,1.9) {$n_z$};
\node[green] at (0.9,2) {$y-su_2$};
       \end{tikzpicture}
       \caption{Notation for the proof of the transversality.\label{fig: fivePointCollision}}
    \end{figure}
We denote by $n_c$ the external unit normal to $c\in\Sigma$, and by $G:\Sigma\to\Sph^{d-1}$ the Gauss map, $G(c)=n_c$.
We also write $\ell_1=x+\R u_1$, $\ell_3 = z+\R u_3$, so that $T(\ell_i)=\ell_{i+1}$.
We will also use subscripts to denote the line at which we consider this subspace of the tangent space.
For example, $\alpha^S_{\ell_2}=dT(V^S_{\ell_1})\subseteq T_{\ell_2}\mathcal{L}$.

\textbf{Transversality with $\alpha^S$}: Since by definition $\alpha^S_{\ell_2}=dT(V^S_{\ell_1})$, then $\alpha^S_{\ell_2}$ corresponds to the wave front of the beam which is obtained from the beam of the parallel rays to $\ell_1$ after the collision. 
The front of the parallel beam is the hyperplane orthogonal to the rays, and hence the curvature operator is the zero operator.
Following this, using the notation of Proposition \ref{prop:billiardCollision}, we conclude that the curvature operator after the collision will be:
\[B=2\inprod{u_2}{n_y}(p_{n_y,u_2}\inv)^* dG_y p_{n_y,u_2}\inv.\]
This is a negative definite operator, since $dG_y$ is positive definite, and $\inprod{u_2}{n_y}<0$, as required.
Note that this case does not use the assumption that this point is part of an m-orbit.

\textbf{Transversality with $\beta^L$}: The subspace $\beta^L_{\ell_2}=dT\inv(V^L_{\ell_3})$ corresponds to a wave front that will form a spherical front emanating from the point $z\in\Sigma$ after its next reflection. 
Before reflection, such a front must have focused on $z$. 
Hence, the curvature operator of this front is a negative definite operator, as needed.
This argument again does not use the assumption that the orbit is an m-orbit.

\textbf{Transversality with $\beta^S$}: The subspace $\beta^S_{\ell_2}=dT\inv(V^S_{\ell_3})$ corresponds to a wave front of a beam that will become parallel after its next reflection at the point $z$.
%
Let $B^-$ denote the curvature of the front right before the collision at $z$.
Then by Proposition \ref{prop:billiardCollision}:
\[0=R^*B^- R +2\inprod{n_z}{u_3}\Big(p_{n_z,u_3}\inv\Big)^*dG_z p_{n_z,u_3}\inv.\]
If we write $B$ for the curvature of the front at $y$, then by Proposition \ref{prop:freeFlight} we have
\[B^-(|z-y|B+I)=B.\]
This can be rewritten as
\[B(|z-y|B^--I)=-B^-.\]
Since $B^-$ is non-degenerate, it follows that $B$ is also non-degenerate, so the result of Proposition \ref{prop:freeFlight} can be rewritten as
\begin{equation}\label{eq:SFFForBetaS}
(B\inv+|z-y|I)\inv = B^- = -2\inprod{n_z}{u_3}R\Big(p_{n_z,u_3}\inv\Big)^* dG_z p_{n_z,u_3}\inv R^*.
\end{equation}
We want $B$ to be a negative definite operator.
 This is equivalent to $B\inv$ being a negative definite operator.
So we solve \eqref{eq:SFFForBetaS} for $B\inv$:
\begin{multline}\label{eq:twoExprForSFFBetaS}
B\inv = R\bigg(-|z-y|I -\frac{1}{2\inprod{n_z}{u_3}}p_{n_z,u_3}(dG_z)\inv p_{n_z,u_3}^* \bigg)R^*=\\
= -|z-y|I -\frac{1}{2\inprod{n_z}{u_3}}R\bigg(p_{n_z,u_3}(dG_z)\inv p_{n_z,u_3}^* \bigg)R^*.
\end{multline}
Observe that $-2\inprod{n_z}{u_3}=|u_3-u_2|$, and $Rp_{n_z,u_3}=p_{n_z,u_2}$ so we only need to check that the operator
\begin{equation}\label{eq:operatorNegativeForBetaS}
\frac{1}{|u_2-u_3|}p_{n_z,u_2}(dG_z)\inv p_{n_z,u_2}^* - |z-y|I,
\end{equation}
is negative definite. 
We show that this indeed follows both from the assumption $S_{11}+S_{22}<0$ (i.e., when the point is part of an m-orbit for $S$) and from the assumption that $L_{11}+L_{22}<0$ (i.e., when the point is part of an m-orbit for $L$).

\textbf{Case of $S_{11}+S_{22}<0$}: 
By Lemma \ref{lem:partialDerivativesS}, we know that 
\begin{gather*}S_{11}(u_2,u_3)+S_{22}(u_1,u_2) = \frac{1}{|u_1-u_2|}p_{u_2,n_y}^*(dG_y)\inv p_{u_2,n_y}+\\
\frac{1}{|u_2-u_3|}p_{u_2,n_z}^*(dG_z)\inv p_{u_2,n_z} - |z-y|I < 0.
\end{gather*}
In this sum, the first term is a positive definite operator.
So if we omit it, then the result will still be negative definite.
But at the same time, we also get the exact operator that is written in \eqref{eq:operatorNegativeForBetaS} (by using the fact that $p_{v,w}^* = p_{w,v})$.
Thus the result follows.

\textbf{Case of $L_{11}+L_{22} < 0$}: 
By Lemma \ref{lem:partialDerivativesL}, we know that:
\begin{gather*}
L_{11}(z,b)+L_{22}(y,z)=\Big(\frac{1}{|z-y|}+\frac{1}{|b-z|}\Big)p_{n_z,u_3}^* p_{n_z,u_3}+\\
+2\inprod{u_3}{n_z}dG_z < 0.
\end{gather*}
If we omit the summand with $|b-z|$, which is positive definite, and use the fact that $2\inprod{u_3}{n_z}=-|u_2-u_3|$, we are left with
\[\frac{1}{|z-y|}p_{n_z,u_3}^*p_{n_z,u_3}<|u_2-u_3|dG_z.\]
The operator on the left-hand side is positive definite, so we can invert this inequality:
\[\frac{1}{|u_2-u_3|}(dG_z)\inv < |z-y|p_{n_z,u_3}\inv \Big(p_{n_z,u_3}\inv\Big)^*.\]
And now we can take congruence with $p_{n_z,u_3}$, which will preserve this inequality:
\[\frac{1}{|u_2-u_3|}p_{n_z,u_3}(dG_z)\inv p_{n_z,u_3}^* < |z-y|I,\]
which means that the operator written in the top line of \eqref{eq:twoExprForSFFBetaS} is negative, as required.

\textbf{Transversality with $\alpha^L$}: 
The subspace $\alpha^L_{\ell_2}=dT(V^L_{\ell_1})$ corresponds to the front obtained by reflecting a beam forming a spherical front centered at $x$. 
At the time of the collision, this front will be a sphere of radius $|x-y|$, so the curvature before the reflection is $\frac{1}{|x-y|}I$.
We use Proposition \ref{prop:billiardCollision} to find the curvature operator $B$ after the collision:
\begin{equation}\label{eq:SFFForAlphaL}
B=\frac{1}{|x-y|}I+2\inprod{u_2}{n_y}\Big(p_{n_y,u_2}\inv\Big)^*dG_y p_{n_y,u_2}\inv.
\end{equation}
We want $B$ to be a negative definite operator.
We show again that this follows both from the assumption $S_{11}+S_{22}<0$ (i.e., when the point is part of an m-orbit for $S$), and from the assumption $L_{11}+L_{22}<0$ (i.e., when the point is part of an m-orbit for $L$).

 \textbf{Case of $S_{11}+S_{22}<0$}: 
By Lemma \ref{lem:partialDerivativesS}, we have 
\begin{gather*}S_{11}(u_1,u_2)+S_{22}(u_0,u_1) = \frac{1}{|u_2-u_1|}p_{u_1,n_y}^* (dG_y)\inv p_{u_1,n_y} + \\
+ \frac{1}{|u_1-u_0|}p_{u_1,n_x}^*(dG_x)\inv p_{u_1,n_x} - |x-y|I < 0 .
\end{gather*}
We can omit the second summand, which is positive definite, and use the fact that $2\inprod{u_2}{n_y}=-|u_1-u_2|$, to get
\[-\frac{1}{2\inprod{u_2}{n_y}}p_{u_1,n_y}^*(dG_y)\inv p_{u_1,n_y} < |y-x|I.\]
Since the left-hand side is a positive definite operator, then we can invert this inequality:
\[\frac{1}{|y-x|}I < -2\inprod{u_2}{n_y}p_{u_1,n_y}\inv dG_y \Big(p_{u_1,n_y}\inv\Big)^*=-2\inprod{u_2}{n_y}\Big(p_{n_y,u_1}\inv\Big)^*dG_y p_{n_y,u_1}\inv,\]
where we used the fact that $p_{v,w}^*=p_{w,v}$.
Now, take congruence with $R\inv$, where $R$ is the orthogonal reflection through $n_y$, restricted from $\set{u_1}^\perp$ to $\set{u_2}^\perp$.
It holds that $Rp_{n_y,u_1} = p_{n_y,u_2}$, and so we get that:
\begin{gather*}\frac{1}{|y-z|}I < -2\inprod{u_2}{n_y}\Big(R\inv \Big)^*\Big(p_{n_y,u_1}\inv\Big)^*dG_y p_{n_y,u_1}\inv R\inv =\\
 -2\inprod{u_2}{n_y}\Big(p_{n_y,u_2}\inv\Big)^*dG_y p_{n_y,u_2}\inv,
 \end{gather*}
and this gives the negativity of the operator of \eqref{eq:SFFForAlphaL}.

\textbf{Case of $L_{11}+L_{22}<0$}: 
By Lemma \ref{lem:partialDerivativesL}, we have
\[L_{11}(y,z)+L_{22}(x,y) = \Big(\frac{1}{|y-x|}+\frac{1}{|z-y|}\Big)p_{n_y,u_2}^*p_{n_y,u_2}+2\inprod{u_2}{n_y}dG_y < 0.\]
We can omit the summand with $|z-y|$ since it is positive definite, and then take congruence with $p_{n_y,u_2}\inv$, which will keep the inequality, and we get
\[\frac{1}{|x-y|}I+2\inprod{u_2}{n_y}\Big(p_{n_y,u_2}\inv\Big)^*dG_y p_{n_y,u_2}\inv < 0,\]
which proves the negativity of the operator in \eqref{eq:SFFForAlphaL}.
Thus, we have constructed a homotopy of Lagrangian subspaces that connects $V^L$ and $V^S$, and under the assumption that the orbit is an m-orbit (for either of the functions), this homotopy is transversal to $\alpha^L$, $\beta^L$, $\alpha^S$, $\beta^S$. 
Therefore, Theorem \ref{thm:geometricAssumptionImpliesMOrbitEq} implies that the m-orbits for $L$ and $S$ are indeed the same. This completes the proof. $\qed$
\begin{appendices}
\section{Derivation of second order derivatives}\label{app:secondOrderDerivatives}
In this appendix, we provide the proofs for Lemmas \ref{lem:partialDerivativesS}, \ref{lem:partialDerivativesL}.
\begin{proof}[Proof of Lemma \ref{lem:partialDerivativesS}]
We compute $S_1$ and $S_{11}$, the computation of $S_2$ and $S_{22}$ is analogous, and deriving the formula for the sum $S_{11}+S_{22}$ is also straightforward.
To compute $S_1(u_2,u_3)\xi$ for some vector $\xi\perp u_2$, we need to differentiate $S(\gamma(t),u_3)$ where $\gamma$ is a smooth curve on the unit sphere, with $\gamma(0)=u_2$ and $\dot{\gamma}(0)=\xi$.
\begin{gather*}
S_1(u_2,u_3)\xi = \frac{d}{dt}\mid_{t=0}S(\gamma(t),u_3)=\frac{d}{dt}\mid_{t=0}\inprod{G\inv(n(\gamma(t),u_3))}{\gamma(t)-u_3}=\\
\inprod{dG_{n_3}\inv \frac{d}{dt}\mid_{t=0} n(\gamma(t),u_3)}{u_2-u_3}+\inprod{G\inv(n_3)}{\xi}.
\end{gather*}
It is straightforward to compute and see that 
\begin{equation}\label{eq:normalDerivative}
\frac{d}{dt}\mid_{t=0}n(\gamma(t),u_3)=\frac{1}{|u_2-u_3|}p_{u_2,n_3}\xi.
\end{equation}
In our case, in the first summand, $u_2-u_3$ is parallel to $n_3$, while $dG_{n_3}\inv$ takes values in $\set{n_3}^\perp$, so the first summand vanishes.
Thus,
\[S_1(u_2,u_3)\xi = \inprod{G\inv(n_3)}{\xi}=\inprod{p_{u_2}G\inv(n_3)}{\xi},\]
where we used the fact that $\xi\in\set{u_2}^\perp$. 
As a result, the linear functional $S_1(u_2,u_3)$ can be identified with the vector $p_{u_2}G\inv(n_3)\in\set{u_2}^\perp$.
This proves the first item.
For the second one, we need to compute (where $\nabla$ denotes the Riemannian connection on $\Sph^{d-1}$)
\begin{gather*}S_{11}(u_2,u_3)\xi = \nabla_\xi p_{u_2}G\inv_{n_3} = p_{u_2}\Big(\frac{d}{dt}\mid_{t=0}G\inv(n(\gamma(t),u_3))-\\
\inprod{G\inv(n(\gamma(t),u_3)}{\gamma(t)}\gamma(t)\Big).
\end{gather*}
The result is (using \eqref{eq:normalDerivative}, and the fact that $p_{u_2}$ eliminates the middle two summands):
\begin{gather*}
S_{11}(u_2,u_3)\xi=p_{u_2}\Big(dG_{n_3}\inv \frac{d}{dt}\mid_{t=0}n(\gamma(t),u_3)-\inprod{dG_{n_3}\inv \frac{d}{dt}\mid_{t=0}n(\gamma(t),u_3)}{u_2}u_2-\\
-\inprod{G\inv(n_3)}{\xi}u_2-\inprod{G\inv(n_3)}{u_2}\xi\Big)=p_{u_2}\Big(\frac{1}{|u_2-u_3|}dG_{n_3}\inv p_{u_2,n_3}\xi\\
-\inprod{G\inv(n_3)}{u_2}\xi\Big).
\end{gather*}
Since $G\inv(n_3)=x_3$, and $dG_{n_3}\inv$ takes values in $\set{n_3}^\perp$, the result is
\[S_{11}(u_2,u_3)=\frac{1}{|u_2-u_3|} p_{n_3,u_2}dG_{n_3}\inv p_{u_2,n_3}\xi-\inprod{x_3}{u_2}\xi,\]
which is the desired result.
In order to compute the mixed partial derivative $S_{12}(u_1,u_2)$, we need to differentiate $S_1(u_1,\gamma(t))$, where $\gamma(t)\in \Sph^{d-1}$ is a curve with $\gamma(0)=u_2$, $\dot{\gamma}(0)=\xi\perp u_2$. 
\begin{gather*}
S_{12}(u_1,u_2)\xi = \frac{d}{dt}\mid_{t=0}S_1(u_1,\gamma(t)) = \frac{d}{dt}\mid_{t=0} p_{u_1} G\inv(n(u_1,\gamma(t))) = \\
=p_{u_1}\frac{d}{dt}\mid_{t=0}G\inv(n(u_1,\gamma(t)))
\end{gather*}
The derivative of $G\inv(n(u_1,\gamma(t)))$ will be a vector in $\set{n_2}^\perp$.
Therefore, we can replace the outer projection with $p_{n_2,u_1}$.
Using equation \eqref{eq:normalDerivative}, we get (we add a minus sign since contrary to equation \eqref{eq:normalDerivative}, here the curve is in the second argument, and not in the first):
\begin{gather*}
S_{12}(u_1,u_2)\xi = p_{n_2,u_1} dG\inv_{n_2}\Big(-\frac{p_{u_1,n_2}\xi}{|u_1-u_2|}\Big)=-\frac{1}{|u_1-u_2|}p_{n_2,u_1}dG\inv_{n_2}p_{u_1,n_2}\xi,
\end{gather*}
which is the required formula, and it is clearly non-degenerate.
\end{proof}

\begin{proof}[Proof of Lemma \ref{lem:partialDerivativesL}]
As in the previous lemma, we give the proofs for $L_1$ and $L_{11}$, and the proofs for $L_2$ and $L_{22}$ are similar.
To compute $L_1(x_2,x_3)\xi$ for some vector $\xi\in T_{x_2}\Sigma$, we need to compute $\frac{d}{dt}\mid_{t=0}L(\gamma(t),x_3)$ for a curve $\gamma(t)\in \Sigma$, $\gamma(0)=x_2$, $\dot{\gamma}(0)=\xi$.
The result is (using the fact that $\xi\in T_{x_2}\Sigma = \set{n_2}^\perp$):
\[L_1(x_2,x_3)\xi = \frac{\inprod{x_2-x_3}{\xi}}{|x_2-x_3|}=-\inprod{u_2}{\xi}=\inprod{-p_{n_2}u_2}{\xi}.\]
So the functional $L_1(x_2,x_3)$ can be identified with the vector $-p_{n_2}u_3\in \set{n_2}^\perp = T_{x_2}\Sigma$.
For the second order partial derivative, we need to compute (where in this time $\nabla$ denotes the Riemannian connection on $\Sigma$):
\begin{gather*}
L_{11}(x_2,x_3)\xi=-\nabla_{\xi} p_{G(x_2)}\frac{x_3-x_2}{|x_3-x_2|}=-p_{n_2}\bigg(\frac{d}{dt}\mid_{t=0}\Big(\frac{x_3-\gamma(t)}{|x_3-\gamma(t)|} - \\
-\inprod{\frac{x_3-\gamma(t)}{|x_3-\gamma(t)|}}{G(\gamma(t))}G(\gamma(t)\Big)\bigg).
\end{gather*}
Start with $\frac{d}{dt}\mid_{t=0}\frac{x_3-\gamma(t)}{|x_3-\gamma(t)|}$. 
It can be checked that the result is 
\[-\frac{\xi}{|x_3-x_2|}+\frac{\inprod{u_2}{\xi}}{|x_3-x_2|}u_2.\]
As a result:
\begin{gather*}
L_{11}(x_2,x_3)\xi = -p_{n_2}\Big(-\frac{\xi}{|x_3-x_2|}+\frac{\inprod{u_2}{\xi}}{|x_3-x_2|}u_2 - \\ \inprod{-\frac{\xi}{|x_3-x_2|}+\frac{\inprod{u_2}{\xi}}{|x_3-x_2|}u_2}{n_2}n_2-\inprod{u_2}{dG_{x_2}\xi}n_2-\inprod{u_2}{n_2}dG_{x_2}\xi\Big).
\end{gather*}
The projection on $\set{n_2}^\perp$ eliminates the two summands that are parallel to $n_2$. 
The first and last summands are already in $\set{n_2}^\perp$.
Thus, we get the required result:
\[L_{11}(x_2,x_3)\xi=\frac{\xi}{|x_3-x_2|}-\frac{\inprod{u_2}{\xi}}{|x_3-x_2|}p_{n_2}u_2+\inprod{u_2}{n_2}dG_{x_2}\xi.\]
As for the sum, from the billiard reflection laws it follows that $p_{n_2}u_1=p_{n_2}u_2$ and $\inprod{u_2}{n_2}=-\inprod{u_1}{n_2}$, so if we sum the expression for $L_{11}(x_2,x_3)$ and $L_{22}(x_1,x_2)$ then we will get
\[\Big(\frac{1}{|x_1-x_2|}+\frac{1}{|x_2-x_3|}\Big)(I-\inprod{u_2}{\cdot}p_{n_2}u_2)+2\inprod{u_2}{n_2}dG_{x_2},\]
and it is immediate to check that on $\set{n_2}^\perp$, we have $I-\inprod{u_1}{\cdot}p_{n_2}u_2=p_{u_2,n_2}p_{n_2,u_2}$, which finishes this computation.
Finally, we compute the mixed partial derivative.
To compute $L_{12}(x_2,x_3)$ we need to differentiate $L_1(x_2,\gamma(t))$ where $\gamma(t)$ is a curve on $\Sigma$ with $\gamma(0)=x_3$.
The result is:
\begin{gather*}
L_{12}(x_2,x_3)\xi = \frac{d}{dt}\mid_{t=0}L_1(x_2,\gamma(t)) = \frac{d}{dt}\mid_{t=0}-p_{n_2}\frac{\gamma(t)-x_2}{|\gamma(t)-x_2|} = \\
=-p_{n_2}\frac{d}{dt}\mid_{t=0}\frac{\gamma(t)-x_2}{|\gamma(t)-x_2|} = -p_{n_2}\Big(\frac{\xi}{|x_3-x_2|} - \frac{\inprod{\xi}{u_2}u_2}{|x_3-x_2|}\Big) = \\
=\frac{-p_{n_2}p_{u_2}\xi}{|x_3-x_2|}.
\end{gather*}
Since $\xi\in T_y\Sigma = \set{n_3}^\perp$, then the inner projection can be replaced with $p_{n_3,u_2}$, and since the inner projection has values in $\set{u_2}^\perp$, the outer projection can be replaced with $p_{u_2,n_2}$, giving the desired formula.
It is non-degenerate since both projections are non-degenerate (the vector $u_2$ is not parallel to either $n_3$ or $n_2$).
\end{proof}
\section{Proof of Sinai-Chernov formula}\label{app:sinaiChernovProof}
In this appendix, we provide a proof for the Sinai-Chernov formula.
Before we begin the proof, let us describe another way to connect a Lagrangian subspace of the space of oriented lines, and the curvature operator of a hypersurface, suggested in \cite{WojtkowskiMaciejP.2004HB}.
As explained in Subsection \ref{subsection:wavefront}, each Lagrangian subspace of $T_\ell \mathcal{L}$, the tangent space to the space of oriented lines at the line $\ell$, induces a local foliation of a neighborhood of points of $\ell$ by hypersurfaces, and the original subspace is then the graph of the curvature operator of those hypersurfaces.
The curvature operator can also be described in the following way. 
Suppose that $\gamma_\varepsilon(t) = a_\varepsilon t + b_\varepsilon$, $|a_\varepsilon| = 1$, is a family of lines that are normal to a given hypersurface $\Sigma$.
This is a family of geodesics in the Euclidean space, and the Jacobi field of this variation is then
\[J(t) = \frac{\partial}{\partial\varepsilon}\mid_{\varepsilon = 0} \gamma_\varepsilon(t).\]
We consider the component of $J$ which is orthogonal to the central ray, $a_0$, $p_{a_0}J$.
Now consider the derivative of this component with respect to time:
\[(p_{a_0}J(t))' = \nabla_{\dot{\gamma_\varepsilon}(t)}p_{a_0}J(t) = \nabla_{p_{a_0}J(t)}\dot{\gamma_\varepsilon}(t),\]
where we used the fact that $\nabla$ is torsion-free, and that the commutator $[p_{a_0}J(t),\dot{\gamma_\varepsilon}(t)]$ vanishes.
Now, the vectors $\dot{\gamma_\varepsilon}(0)$ are a normal field to $\Sigma$, so by definition, $\nabla_{p_{a_0}J(0)}\dot{\gamma_\varepsilon}(t) = B(p_{a_0}J(0))$, where $B$ denotes the curvature operator of $\Sigma$.
Then the conclusion is that 
\begin{equation}\label{eq:JacobiAndCurvature}
(p_{a_0}J)'(0) = B(p_{a_0}J(0)).
\end{equation}
\begin{proof}[Proof of Proposition \ref{prop:billiardCollision}]
Consider a variation of the line $\ell$ to which $Y$ is the tangent space, $\gamma_\varepsilon(t)=u_{1, \varepsilon} t + y_\varepsilon$, with $|u_{1,\varepsilon}|=1$, $y_\varepsilon\in\Sigma$ and $y_0=y$ is the collision point.
Write $u_1$ for $u_{1,0}$, and $\xi = \frac{\partial}{\partial\varepsilon}\mid_{\varepsilon = 0} u_{1,\varepsilon}$, $\eta = \frac{\partial}{\partial\varepsilon}\mid_{\varepsilon = 0} y_\varepsilon$ for the derivatives.
Observe that since $y_\varepsilon\in\Sigma$ then $\eta\in T_{y}\Sigma = \set{n_{y}}^\perp$ (where $n_x$ denotes the outer unit normal at a point $x\in\Sigma$).
Then 
\[p_{u_1}J(0) = p_{u_1}\eta = p_{n_{y},u_1}\eta,\]
\[(p_{u_1}J)'(0) = p_{u_1}\xi=\xi,\]
where we used the fact that $|u_{1,\varepsilon}|=1\implies \xi\perp u_1$.
As a result, \eqref{eq:JacobiAndCurvature} gives
\begin{equation}\label{eq:alphaBetaSinai}
\xi = Bp_{n_{y},u_1}\eta.
\end{equation}
Denote the quantities that refer to the outgoing rays with a $+$ superscript, see Figure \ref{fig:SinaiChernovDemonstration}.
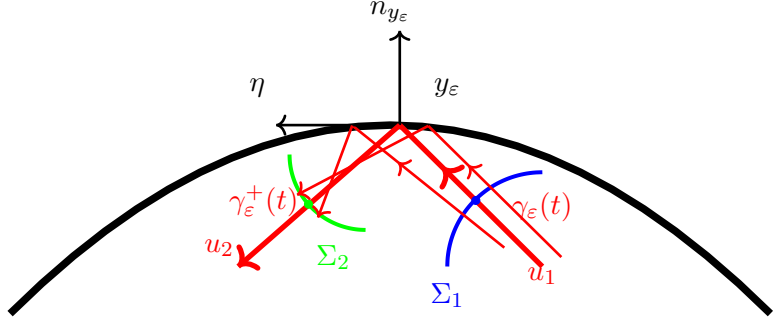
\begin{figure}
	\centering
	\begin{tikzpicture}[scale = 2.5]
	\begin{scope}[decoration={
markings,
mark = at position 0.7 with {\arrow{>}}}
]
	\draw[line width = 1mm, domain = -2:2] plot(\x,1-0.25*\x*\x);
\draw[domain = 90:180, smooth, variable = \x, blue,line width = 1.5pt] plot({0.8+0.5*cos(\x)},{0.25+0.5*sin(\x)});
\draw[red, line width = 2pt, postaction = {decorate}](0.8,0.25)--(0.05,1);
\draw[red, line width = 1pt, postaction = {decorate} ](0.9,0.3)--(0.2,0.996);
\draw[red, line width = 1pt, postaction = {decorate} ](0.6,0.35)--(-0.2,0.998);

\draw[domain = 180:270, smooth, variable = \x, green,line width = 1.5pt] plot({-0.43++0.3+0.4*cos(\x)},{0.58+0.88*0.3+0.4*sin(\x)});
\end{scope}
	\begin{scope}[decoration={
markings,
mark = at position 1 with {\arrow{>}}}
]
\draw[black, line width = 1pt, postaction = {decorate}](0,1)--(-0.6,1);
\draw[black, line width = 1pt, postaction = {decorate}](0.05,1)--(0.05,1.5);
\draw[red, line width = 2pt, postaction = {decorate}](0.05,1)--(-0.8,0.25);
\draw[red, line width = 1pt, postaction = {decorate}](0.2,0.996)--(-0.49,0.63);
\draw[red, line width = 1pt, postaction = {decorate}](-0.2,0.998)--(-0.38,0.51);
\end{scope}
\node[red] at (0.8,0.2) {$u_1$};
\node[blue] at (0.3,0.1) {$\Sigma_1$};
\node[black] at (0.3, 1.2) {$y_\varepsilon$};
\node[black] at (-0.7, 1.2) {$\eta$};
\node[green] at (-0.3,0.3) {$\Sigma_2$};
\node[red] at (-0.9,0.35) {$u_2$};
\node[black] at (0,1.6) {$n_{y_\varepsilon}$};
\node[below,red] at (0.8,0.7) {$\gamma_\varepsilon(t)$};
\node[red,left] at (-0.43,0.6) {$\gamma_\varepsilon^+(t)$};
\tkzDefPoint(0.45,0.6){A};
\tkzDrawPoint[blue](A);
\tkzDefPoint(-0.43,0.58){B}
\tkzDrawPoint[green](B);
	\end{tikzpicture}
	\caption{Reflection of wave fronts. The hypersurface $\Sigma_1$ is orthogonal to the variation $\gamma_\varepsilon(t)$.
	The hypersurface $\Sigma_2$ is orthogonal to the reflection of this variation, $\gamma_\varepsilon^+(t)$. \label{fig:SinaiChernovDemonstration}}
\end{figure}
We then have
$\gamma_\varepsilon^+(t)= u_{1,\varepsilon}^+ t + y_\varepsilon$, where $u_{1,\varepsilon}^+ = u_{1,\varepsilon} - 2\inprod{u_{1,\varepsilon}}{n_{y_\varepsilon}}n_{y_\varepsilon}$, and:
\begin{gather*}
\xi^+ = \frac{\partial }{\partial \varepsilon}\mid_{\varepsilon=0} u_{1,\varepsilon}^+= \xi - 2\inprod{\xi}{n_{y}}n_{y} - 2\inprod{u_1}{dG_{y}\eta}n_{y}- \\
-2\inprod{u_1}{n_{y}}dG_{y}\eta.
\end{gather*}
So in this case, (here $u_2=u_1^+$ is the direction of the central ray after reflection) $p_{u_2}J^+(0) = p_{u_2}\eta=p_{n_{y},u_2}\eta$, and $(p_{u_2}J^+)'(0) = p_{u_2}\xi^+$.
Let us compute the last expression more explicitly.
\begin{gather*}
(p_{u_2}J^+)'(0)=p_{u_2}(\xi-2\inprod{\xi}{n_{y}}n_{y}-2\inprod{u_1}{dG_{y}\eta}n_{y}-2\inprod{u_1}{n_{y}}dG_{y}\eta) = \\
=p_{u_2}R^*\xi - 2p_{u_2}(\inprod{u_1}{dG_{y}\eta}n_{y}+\inprod{u_1}{n_{y}}dG_{y}\eta)
\end{gather*}
Since $\xi\perp u_1$, then $R^*\xi\perp u_2$, so $p_{u_2}R^*\xi=R^*\xi$. 
Also, since $u_1$, $u_2$ are reflections about the hyperplane $\set{n_{y}}^\perp$, then $\inprod{u_1}{dG_{y}\eta}=\inprod{u_2}{dG_{y}\eta}$ and $\inprod{u_1}{n_{y}}=-\inprod{u_2}{n_{y}}$.
Hence,
\begin{gather*}
(p_{u_2}J^+)'(0)=R^*\xi+2\inprod{u_2}{n_{y}}p_{u_2}\Big(dG_{y}\eta-\frac{\inprod{u_2}{dG_{y}\eta}}{\inprod{u_2}{n_{y}}}n_{y}\Big).
\end{gather*}
The argument of $p_{u_2}$ in the formula above is a vector which is orthogonal to $u_2$, whose orthogonal projection on $\set{n_{y}}^\perp$ is $dG_{y}\eta$, and since $dG_{y}\eta\in\set{n_{y}}^\perp$ then we can simplify:
\[(p_{u_2}J^+)'(0)=R^*\xi+2\inprod{u_2}{n_{y}}p_{u_2,n_{y}}\inv dG_{y}\eta.\]
Now, as before \eqref{eq:JacobiAndCurvature} gives $(p_{u_2}J^+)'(0)=B^+ p_{u_2}J^+(0) = B^+p_{n_{y},u_2}\eta$.
If we also use \eqref{eq:alphaBetaSinai}, then we get
\[B^+p_{n_{y},u_2}\eta = R^*Bp_{n_{y},u_1}\eta + 2\inprod{u_2}{n_{y}}p_{u_2,n_{y}}\inv dG_{y}\eta.\]
And as a result, we get the required formula:
\begin{gather*}
B^+ = R^*Bp_{n_{y},u_1}p_{n_{y},u_2}\inv + 2 \inprod{u_2}{n_{y}}p_{u_2,n_{y}}\inv dG_{y}p_{n_{y},u_2}\inv = \\
=R^*BR + 2\inprod{u_2}{n_{y}}(p_{n_{y},u_2}\inv)^*dG_{n_{y}}p_{n_{y},u_2}\inv.
\end{gather*}
\end{proof}

\end{appendices}
\bibliography{bibliography}
\bibliographystyle{abbrv}

\end{document}